\documentclass[11pt,draft]{amsart}
\usepackage{amssymb,amsfonts,amsmath,enumerate,xypic}
\usepackage[mathscr]{euscript}

\newcommand{\Z}{\mathbb{Z}}   
\newcommand{\nn}{\overline{\mathbb{Z}}_{+}}

\newcommand{\I}{\mathscr{I}}
\renewcommand{\P}{\mathscr{P}}
\newcommand{\R}{\mathscr{R}}  
\newcommand{\f}{\mathfrak{f}}
\def\ndash{--}

\providecommand{\Ind}{\mathop{\mathrm{Ind}}\nolimits} 
\providecommand{\Res}{\mathop{\mathrm{Res}}\nolimits} 
\providecommand{\val}{\mathop{\mathrm{val}}\nolimits} 

\renewcommand{\aa}{\mathtt{a}}
\newcommand{\bb}{\mathtt{b}}
\newcommand{\cc}{\mathtt{c}}
\newcommand{\dd}{\mathtt{d}}
\newcommand{\ii}{\mathtt{i}}
\newcommand{\RR}{\mathtt{R}}
\newcommand{\TT}{\mathtt{T}}
\newcommand{\XX}{\mathtt{X}}

\def\disp#1#2{\overset{\makebox[0.75\width]{$#2$}}{\makebox[\width]{$#1$}}}
\def\disq#1#2{\overset{\makebox[1\width]{$#2$}}{\makebox[\width]{$#1$}}}

\theoremstyle{plain}
\newtheorem{theorem}{Theorem}[section]
\newtheorem{lemma}[theorem]{Lemma}
\newtheorem{proposition}[theorem]{Proposition}
\newtheorem{corollary}[theorem]{Corollary}

\theoremstyle{definition}
\newtheorem{definition}[theorem]{Definition}
\newtheorem{remark}[theorem]{Remark}

\begin{document}

\title[Branching rules for unramified principal series representations]{Branching rules for unramified principal series representations of $\mathrm{GL}(3)$ over a $p$-adic field}

\thanks{This research was supported by grants from NSERC and the Faculty of Science~of the University of Ottawa.  The first author would also like to acknowledge the support of the Centre de Recherches Math\'ematiques while at the University of~Ottawa.}

\author{Peter S. Campbell}
\email{peter.campbell@bristol.ac.uk}
\address{
Department of Mathematics,
University of Bristol,
Bristol
BS8 1TW,
UK}

\author{Monica Nevins}
\email{mnevins@uottawa.ca}
\address{
Department of Mathematics \and Statistics,
University of Ottawa,
Ottawa, Ontario
K1N 6N5,
Canada}

\begin{abstract}
On restriction to the maximal compact subgroup $\mathrm{GL}(3,\R)$, an unramified principal series representation of the $p$-adic group $\mathrm{GL}(3,F)$ decomposes into a direct sum of finite-dimensional irreducibles each appearing with finite multiplicity.  We describe a coarser decomposition into components which, although reducible in general, capture the equivalences between the irreducible constituents.
\end{abstract}


\maketitle


\section{Introduction}

The aim of this paper is to investigate the relationship between the representation theory of a $p$-adic group $G$ and its maximal compact subgroups $K$.  Given an admissible representation of $G$, its restriction to $K$ decomposes as a direct sum of smooth irreducible representations of $K$ each with finite multiplicity.  The problem of describing this decomposition when $G=\mathrm{GL}(2,F)$ and $K=\mathrm{GL}(2,\R)$, for $F$ a non-archimedean local field of odd residual characteristic and its ring of integers $\R$, was extensively studied by Silberger \cite{sil1} and  Casselman \cite{cas2} with the restriction on the characteristic removed.  Further, the case of the principal series representations for $G=\mathrm{SL}(2,F)$ was considered by the second author in \cite{ne}.

We are interested in $G=\mathrm{GL}(3,F)$ and its unramified principal series representations; the ramified case will be treated in a separate paper.  The restriction to $K=\mathrm{GL}(3,\R)$ of any unramified principal series representation is simply the permutation representation over the subgroup $B$ of upper triangular matrices in $K$.  In particular, this contains the pull-back of the corresponding permutation representation for the group $\mathrm{GL}(3,\f)$, defined over the residue field $\f$ of $F$, and the decomposition of this is well known \cite{st}: each irreducible constituent can be expressed as an alternating sum of permutation representations over certain standard parabolic subgroups in $\mathrm{GL}(3,\f)$.

Our approach is generalise this by considering representations $V_{\cc}$, indexed by triples $\cc=(c_{1},c_{2},c_{3})$ with $0\leq c_{1},c_{2}\leq c_{3}\leq c_{1}+c_{2}$, which are expressible in terms of permutation representations over compact open subgroups containing $B$.  By determining the double coset structure of $K$ we are able to calculate the intertwining number $\I(V_{\cc},V_{\dd})$ for any two such components; that is, the dimension of the space of $K$-homomorphisms between $V_{\cc}$ and $V_{\dd}$.  Although $V_{\cc}$ is irreducible when $c_{3}=c_{1}+c_{2}$ or $\max\{c_{1},c_{2}\}$, we find that it is reducible in general with $\I(V_{\cc},V_{\cc})$ depending on the order of the residue field.  However, it transpires that two components are either completely equivalent or contain no common constituents. In the final section, we present an application of our results to a family of virtual representations defined by Lees \cite{lees} as analogues of the Steinberg representation.


\section{Principal series representations}

Let $F$ be a non-archimedean local field with ring of integers $\R$ and residue field $\mathfrak{f}$.  We will assume that $\mathfrak{f}$ has odd characteristic and order $q$.  If $\pi$ denotes a conductor of $F$ then the maximal ideal of $\R$ is $\P=\pi\R$.  For each positive integer $n\in \mathbb{Z}_{+}$ we define $\P^{n} = \{ x\in F : \val(x)\geq n\}$ where $\val$ is the discrete valuation on $F$ normalised so that $\val (\pi) = 1$.

Let $\mathbb{G}=\mathrm{GL}(3)$, then $\mathbb{G}(F)= \mathrm{GL}(3,F)$ is a locally compact group with maximal compact open subgroup $K=\mathbb{G}(\R) = \mathrm{GL}(3,\R)$. Indeed, the topology on $\mathbb{G}(F)$ has a neighbourhood base about the identity given by the compact open subgroups $K_{n} = 1 + M_{3,3}(\P^{n})$ for $n\in \Z_{+}$.  Further, let $\mathbb{B}$ be the subgroup of upper triangular matrices and recall that $\mathbb{B}$ decomposes as $\mathbb{B}=\mathbb{T}\mathbb{U}$ where $\mathbb{T}$ is the subgroup of diagonal matrices and $\mathbb{U}$ is the subgroup of upper unitriangular matrices.  We will denote by $B$, $T$ and $U$ the subgroups $\mathbb{B}(\R)$, $\mathbb{T}(\R)$ and $\mathbb{U}(\R)$ of $K$ respectively.

Given a character $\chi$ of $\mathbb{T}(F)$ we may extend it to a character of $\mathbb{B}(F)$, again denoted $\chi$, by defining it to be trivial on $\mathbb{U}(F)$.  The corresponding principal series representation of $\mathbb{G}(F)$ is the induced representation $\Ind_{\mathbb{B}(F)}^{\mathbb{G}(F)}\chi$ consisting of the space smooth functions
\begin{equation*}
V = \{ f\in C^{\infty}(\mathbb{G}(F)) : f(bg) = \chi(b)|b|f(g) \text{ for all }g\in \mathbb{G}(F), b\in\mathbb{B}(F)\}
\end{equation*}
with the action of $\mathbb{G}(F)$ given by right translation.  The normalization factor $|b|$ is introduced to ensure that $\Ind_{\mathbb{B}(F)}^{\mathbb{G}(F)}\chi \simeq \Ind_{\mathbb{B}(F)}^{\mathbb{G}(F)}\chi'$ whenever $\chi$ and $\chi'$ lie in the same orbit under the Weyl group $W$ of $\mathbb{G}$ (see \cite{car}*{Theorem 3.3}, for example).

We will be interested in the restriction of the principal series representation $V$ to the maximal compact subgroup $K$.  As $\mathbb{G}(F)=K\mathbb{B}(F)$ and $B = \mathbb{B}(F) \cap K$, Mackey theory implies that
\begin{equation*}
\Res_{K}^{\mathbb{G}(F)} V  \simeq \Ind_{B}^{K} \Res_{B}^{\mathbb{B}(F)} \chi.
\end{equation*}
This can be interpreted as the principal series representation of $K$ obtained from the character $\Res_{T}^{\mathbb{T}(F)} \chi$ of $T$. The first step towards decomposing the restriction into irreducibles is the following result regarding the principal congruence subgroups $K_{n}$ of $K$.

\begin{lemma} The subspaces $V^{K_{n}}$ of vectors fixed under the action of $K_{n}$ are $K$-stable and finite-dimensional.  They are non-zero if and only if $K_{n}\cap T\subseteq \ker(\chi)$, in which case $\chi$ extends trivially to a character of $BK_{n}$ and
\begin{equation*}
V^{K_{n}} = \Ind_{BK_{n}}^{K} \chi
\end{equation*}
where both sides are viewed as $K$-representations.
\end{lemma}

In this paper we will be concerned the unramified principal series; that is, the case where the restriction of $\chi$ to $T$ is the trivial character $1$.  Here we obtain the permutation representation
\begin{equation*}
\Res_{K}^{\mathbb{G}(F)} V \simeq \Ind_{B}^{K} 1
\end{equation*}
and for each $n\in \Z_{+}$
\begin{equation*}
V^{K_{n}} = \Ind_{BK_{n}}^{K} 1.
\end{equation*}


\section{A decomposition} \label{sec:defn}

The filtration of $K$ by congruence subgroups allows us to decompose the representation $V$ into a direct sum of finite-dimensional $K$-invariant subspaces
\begin{equation*}
V \simeq \bigoplus_{n=0}^{\infty} V^{K_{n}}/V^{K_{n-1}}.
\end{equation*}
However, these quotients are far from being irreducible in general so we will consider a finer filtration of $K$ obtained from certain compact open subgroups $C_{\cc}$.

Define the partially ordered set
\begin{equation*}
\TT = \{ \cc=(c_{1},c_{2},c_{3})\in \mathbb{Z}^{3} : 0 \leq c_{1}, c_{2} \leq c_{3} \leq c_{1}+c_{2}\}
\end{equation*}
with order given by $\cc \succeq \dd$ if and only if $c_{i}\geq d_{i}$ for each $i$.  We associate to each triple $\cc=(c_{1},c_{2},c_{3})\in \TT$ a compact open subgroup $C_{\cc}$ of $K$ by defining
\begin{equation*}C_{\cc}=\left[ \begin{array}{ccc}
\R         & \R         & \R \\
\P^{c_{1}} & \R         & \R \\
\P^{c_{3}} & \P^{c_{2}} & \R
\end{array}\right] \cap K\end{equation*}
and note that $C_{\cc} \subseteq C_{\dd}$ if and only if $\cc \succeq \dd$.
Consequently, if for each $\cc\in \TT$ we set
\begin{equation*}
U_{\cc} = \Ind_{C_{\cc}}^{K} 1
\end{equation*}
then $U_{\dd}$ arises as a subrepresentation of $U_{\cc}$ precisely when $\dd \preceq \cc.$  Thus we can consider the quotient
\begin{equation*}V_{\cc} = U_{\cc}/ \sum_{\dd \prec \cc} U_{\dd}.\end{equation*}
In particular, since $BK_{n} = C_{(n,n,n)}$, we see that
\begin{equation*}V^{K_{n}}={\bigoplus_{\cc \preceq (n,n,n)}} V_{\cc}.\end{equation*}

Our aim is to determine the reducibility of and equivalences between the $V_{\cc}$.  To achieve this we first give a description of $V_{\cc}$ as an alternating sum in the Grothendieck group $\mathscr{K}_{0}(K)$ of $K$.  Recall that $\mathscr{K}_{0}(K)$ is the abelian group generated by the isomorphism classes $[V]$ of finitely-generated representations $V$ of $K$ together with the relations $[V\oplus U] = [V] + [U]$ and $[V/U]=[V]-[U]$.

We begin with some notation.  Let $\cc=(c_{1},c_{2},c_{3})\in \TT$ and, if $c_{1}$ and $c_{2}$ are both non-zero, for each $1\leq i \leq 3$ define
\begin{equation*}\cc_{\{i\}} = (c_{1}-\delta_{i,1}, c_{2}-\delta_{i,2}, c_{3}-\delta_{i,3})\end{equation*}
where $\delta_{i,j}=1$ if $i=j$ and $0$ otherwise.  If $c_{1}=0$ then $\cc=(0,c_{2},c_{2})$ and we only consider $\cc_{\{3\}} = (0,c_{2}-1,c_{2}-1)$.  Similarly, if $c_{2}=0$ then we only have $\cc_{\{3\}} = (c_{1}-1,0,c_{1}-1)$.
The set of all triples in $\TT$ lying immediately below $\cc$ is then $\{ \cc_{\{i\}} : i\in S_{\cc}\}$ where $S_{\cc} = \{ i : \cc_{\{i\}} \in \TT\}$.  In particular, this means that
\begin{equation*}V_{\cc} = U_{\cc} / \sum_{i\in S_{\cc}} U_{\cc_{\{i\}}}.\end{equation*}

Further, let $\cc_{\emptyset} = \cc$ and for each non-empty $I\subseteq S_{\cc}$ define
\begin{equation*}\cc_{I} = \max \{ \dd \in \TT : \dd \preceq \cc_{\{i\}} \text{ for all } i\in I\}.\end{equation*}
For example, if $\cc=(2,3,4)$ then $\cc_{\{1,2\}} = (1,2,3)$ since $(1,2,4)\notin \TT.$

\begin{lemma}
For each $I,J\subseteq S_{\cc}$ we have $U_{\cc_{I}}\cap U_{\cc_{J}} = U_{\cc_{I\cup J}}$.
\end{lemma}

\begin{lemma} \label{lem:distributive}
For any $\cc\in \TT$ with $S_{\cc}=\{1,2,3\}$
\begin{equation*}(U_{\cc_{\{1\}}} + U_{\cc_{\{2\}}} ) \cap U_{\cc_{\{3\}}} = U_{\cc_{\{1,3\}}} + U_{\cc_{\{2,3\}}}.\end{equation*}
\end{lemma}
\begin{proof}
This follows from \cite{cam}*{Lemma 13} since $C_{\cc_{\{i\}}} C_{\cc_{\{3\}}}  = C_{\cc_{\{3\}}} C_{\cc_{\{i\}}} = C_{\cc_{\{i,3\}}}$ for each $i$.
\end{proof}

\begin{proposition} \label{prop:altsum}
For any $\cc\in \TT$
\begin{equation*}[V_{\cc}] = \sum_{I\subseteq S_{\cc}} (-1)^{|I|} [U_{\cc_{I}}].\end{equation*}
\end{proposition}
\begin{proof}
First note that if $S_{\cc}=\{i\}$ then $V_{\cc} = U_{\cc}/U_{\cc_{i}}$ so clearly
\begin{equation*}[V_{\cc}] = [U_{\cc}] - [U_{\cc_{\{i\}}}].\end{equation*}
Further, if $S_{c}=\{i,j\}$ then $U_{\cc_{\{i\}}} + U_{\cc_{\{j\}}} = X \oplus U_{\cc_{\{j\}}}$ where $X=U_{\cc_{\{i\}}}/(U_{\cc_{\{i\}}} \cap U_{\cc_{\{j\}}})$ and $U_{\cc_{\{i\}}} \cap U_{\cc_{\{j\}}}= U_{\cc_{\{i,j\}}}$.  This gives $[U_{\cc_{\{i\}}} + U_{\cc_{\{j\}}}]= [U_{\cc_{\{i\}}}] + [U_{\cc_{\{j\}}}] - [U_{\cc_{\{i,j\}}}]$ and $V_{\cc} = U_{\cc}/(U_{\cc_{\{i\}}} + U_{\cc_{\{j\}}})$ implies that
\begin{equation*}[V_{\cc}] = [U_{\cc}] - [U_{\cc_{\{i\}}}] - [U_{\cc_{\{j\}}}] + [U_{\cc_{\{i,j\}}}].\end{equation*}
Finally, if $S_{\cc} = \{1,2,3\}$ then $U_{\cc_{\{1\}}} + U_{\cc_{\{2\}}} + U_{\cc_{\{3\}}} = X \oplus U_{\cc_{\{3\}}}$ where on this occasion $X=(U_{\cc_{\{1\}}} + U_{\cc_{\{2\}}})/ ((U_{\cc_{\{1\}}} + U_{\cc_{\{2\}}})\cap U_{\cc_{\{3\}}})$.  From Lemma~\ref{lem:distributive} we know that $(U_{\cc_{\{1\}}} + U_{\cc_{\{2\}}} ) \cap U_{\cc_{\{3\}}} = U_{\cc_{\{1,3\}}} + U_{\cc_{\{2,3\}}}$ so using the same argument as before we see that
\begin{equation*}[X] = [U_{\cc_{\{1\}}}] + [U_{\cc_{\{2\}}}] - [U_{\cc_{\{1,2\}}}] - [U_{\cc_{\{1,3\}}}] - [U_{\cc_{\{2,3\}}}] + [U_{\cc_{\{1,2,3\}}}].\end{equation*}
Hence, $V_{\cc} = U_{\cc}/(U_{\cc_{\{1\}}} + U_{\cc_{\{2\}}} + U_{\cc_{\{3\}}})$ gives
\begin{equation*}[V_{\cc}] = [U_{\cc}] - [U_{\cc_{\{1\}}}] - [U_{\cc_{\{2\}}}] - [U_{\cc_{\{3\}}}] + [U_{\cc_{\{1,2\}}}] + [U_{\cc_{\{1,3\}}}] + [U_{\cc_{\{2,3\}}}] - [U_{\cc_{\{1,2,3\}}}]\end{equation*}
as required.
\end{proof}

The space of $K_{1}$-fixed vectors in $V$
\begin{equation*}V^{K_{1}}=\Ind_{C_{(1,1,1)}}^{K} 1\end{equation*}
is the pull-back to $K$ of the permutation representation $\Ind_{\mathbb{B}(\mathfrak{f})}^{\mathbb{G}(\mathfrak{f})} 1$ so its decomposition into irreducibles is well known.  Specifically,
\begin{equation*}V^{K_{1}} = V_{(0,0,0)} \oplus V_{(0,1,1)} \oplus V_{(1,0,1)} \oplus V_{(1,1,1)}\end{equation*}
where $[V_{(0,0,0)}]=[U_{(0,0,0)}]$ is the trivial representation; $[V_{(0,1,1)}]=[U_{(0,1,1)}]-[U_{(0,0,0)}]$ and $[V_{(1,0,1)}]=[U_{(1,0,1)}]-[U_{(0,0,0)}]$ are the equivalent irreducible constituents; and $[V_{(1,1,1)}]=[U_{(1,1,1)}]-[U_{(0,1,1)}]-[U_{(1,0,1)}]+[U_{(0,0,0)}]$ corresponds to the Steinberg representation which is irreducible with multiplicity $1$.

More generally, we can use Proposition~\ref{prop:altsum} to calculate the intertwining number between two quotients $V_{\cc}$ and $V_{\dd}$ as an alternating sum involving the intertwining numbers between various $U_{\cc}$ and $U_{\dd}$
\begin{equation*}
\I(V_{\cc},V_{\dd}) = \sum_{I\subseteq S_{\cc},\ J\subseteq S_{\dd}} (-1)^{|I|+|J|} \I(U_{\cc_{I}},U_{\dd_{J}}).
\end{equation*}
However, since $U_{\cc_{I}}$ and $U_{\dd_{J}}$ are the permutation representations on $C_{\cc_{I}}$ and $C_{\dd_{J}}$ respectively, we have $\I(U_{\cc_{I}},U_{\dd_{J}}) = |C_{\cc_{I}} \backslash K / C_{\dd_{J}}|$, the number of $(C_{\cc_{I}},C_{\dd_{J}})$-double cosets in $K$.  Thus we obtain the following.

\begin{corollary}  \label{cor:altsum}
Let $\cc, \dd\in \TT$, then
\begin{equation*}
\I(V_{\cc},V_{\dd}) = \sum_{I\subseteq S_{\cc},\ J\subseteq S_{\dd}} (-1)^{|I|+|J|} |C_{\cc_{I}} \backslash K / C_{\dd_{J}}|.
\end{equation*}
\end{corollary}

Finally, we note that Proposition~\ref{prop:altsum} also allows us to determine the dimensions of the $V_{\cc}$ for $\cc\in \TT$ with $c_{3}>1$.  If we let $\left|\cc\right| = c_{1}+c_{2}+c_{3}$ then
\begin{equation*}\dim U_{\cc} = [K:C_{\cc}] = \left\{\begin{array}{ll}
(q+1)(q^2+q+1)q^{\left|\cc\right|-3} & \text{if $c_{1},c_{2}>0$};\\
(q^2+q+1)q^{\left|\cc\right|-2} & \text{if $c_{1}=0$ or $c_{2}=0$}
\end{array}\right.
\end{equation*}
so, writing $c_{3}=c_{1}+c_{2}-k$ for some $0\leq k \leq \min\{c_{1},c_{2}\}$, we have
\begin{equation*}
\dim V_{\cc} = \left\{\begin{array}{ll}
(q-1)(q+1)(q^2+q+1)q^{\left|\cc\right|-4}& \text{if $k=0$};  \\
(q-1)(q-2)(q+1)(q^2+q+1)q^{\left|\cc\right|-5} & \text{if $k=1$}; \\
(q-1)^3(q+1)(q^2+q+1)q^{\left|\cc\right|-6} & \text{if $1<k<\min\{c_{1},c_{2}\}$};\\
(q-1)^2(q+1)(q^2+q+1)q^{\left|\cc\right|-5}   & \text{if $k=\min\{c_{1},c_{2}\}$}.
\end{array}\right.
\end{equation*}


\section{(B,B)-double cosets} \label{sec:bb}

It is clear from Corollary \ref{cor:altsum} that we need to describe the $(C_{\cc},C_{\dd})$-double coset structure of $K$.  However, before tackling the general case we examine the double cosets of $K$ with respect to the subgroup $B$ of upper triangular matrices.   These, and indeed the double cosets in the case where $\cc=(c,c,c)=\dd$, have recently been described by Onn, Prasad and Vaserstein \cite{opv}.

Let $W=\{ 1 , s_{1}, s_{2}, s_{1}s_{2}, s_{2}s_{1}, w_{0}\}$ denote the group of permutation matrices in K where $s_{i}$ corresponds to the transposition $(i\ \ i+1)$ and $w_{0}$ is the element of maximal length.  From the Bruhat decomposition of $\mathrm{GL}(3,\mathfrak{f})$ we can choose our $(B,B)$-double coset representatives to be of the form $wk$ for some $w\in W$ and $k\in K_{1}$.  If we let $U^{-}$ denote the subgroup of lower unitriangular matrices in $K$, then the decomposition $K_{1}=(K_{1}\cap U^{-})(K\cap B)$ means that we may take $k\in K_{1}\cap U^{-}$.  Further, we have $U^{-}=V_{w}^{-}V_{w}$ where
\begin{equation*}V_{w} = \left\{ [k_{ij}]\in U^{-} : k_{ij}=0 \text{ if } w(i)<w(j)\right\}\end{equation*}
and
\begin{equation*}V_{w}^{-} = \left\{ [k_{ij}]\in U^{-} : k_{ij}=0 \text{ if } w(i)>w(j)\right\}.\end{equation*}
Thus, writing $k=k_{1}k_{2}$ with $k_{1}\in V_{w}^{-}$, $k_{2}\in V_{w}$ we see that $BwkB = Bwk_{2}B$ since $wk_{1}w^{-1}\in U$.  We have therefore obtained the following special case of \cite{hill}*{Proposition 2.6}.

\begin{lemma} \label{lem:BBhill}
Every $(B,B)$-double coset representative in $K$ can be chosen of the form $wk$ for some $w\in W$ and $k\in V_{w}$.
\end{lemma}

While Lemma~\ref{lem:BBhill} shows that there is exactly one double coset corresponding to $w_{0}$, it does not give any information about the double cosets lying in the Iwahori subgroup $BK_{1}$.
Let $\nn = \mathbb{Z}_{+} \cup \{\infty\}$ with the convention that $a<\infty$ and $\infty+a = \infty-a=\infty$  for every $a\in \mathbb{Z}_{+}$.  Define the set of triples
\begin{equation*}\TT^{\infty} = \{ (a_{1},a_{2},a_{3}) \in \nn^{3} : a_{1},a_{2}\leq a_{3}\}\end{equation*}
and for each $\aa \in \TT^{\infty}$, $x\in \R^{\times}$ consider the element
\begin{equation*}t_{\aa,x} =
\left[\begin{array}{ccc}
1            & 0           & 0 \\
\pi^{a_{1}}  & 1           & 0 \\
\pi^{a_{3}}x & \pi^{a_{2}} & 1
\end{array}\right]\end{equation*}
where we take $\pi^{\infty} = 0$ and so set $\val(0)=\infty$.  In the following we identify $\R/\P^{i}$ with a set of representatives in $\R$ chosen so that they contain the representatives corresponding to $\R/\P^{j}$ for each $j<i$.

\begin{proposition} \label{prop:BBiwahori}
A complete set of $(B,B)$-double coset representatives in $BK_{1}$~is
\begin{equation*}
\RR^{1} = \left\{t_{\aa,x} :
\aa\in\TT^{\infty}, x\in \XX^{\aa}\right\}
\end{equation*}
where
\begin{equation*}\XX^{\aa} = \left\{ \begin{array}{ll}
\{ 1 \}
& \text{if $a_{3} = \infty$}; \\
\left(\R/\P^{\min\{a_{1},a_{2},a_{3}-a_{1},a_{3}-a_{2}\}}\right)^{\times}
& \text{if $a_{1}+a_{2} \neq a_{3}$ and $a_{3}< \infty$}; \\
\bigcup_{i=0}^{\infty} (1+\pi^{i} \R^{\times}) \cap (\R/\P^{\min\{a_{1},a_{2}\}+i})^{\times}
& \text{if $a_{1}+a_{2} = a_{3}$ and $a_{3} < \infty$}.
\end{array}\right.
\end{equation*}
\end{proposition}
\begin{proof}
From Lemma~\ref{lem:BBhill} we can choose our representative $t=[t_{ij}]$ to lie in $K_{1}\cap U^{-}$.  Indeed, since left and right multiplication by elements of $B$ allows us to add multiples of $t_{31}$ to $t_{21}$ and $t_{32}$, we may assume that the lower triangular entries of $t$ are such that
$\max\{ \val(t_{21}),\val(t_{32})\} \leq \val(t_{31}).$
Further, conjugating by elements of $T$ enables us to independently scale $t_{21}$ and $t_{32}$ by elements of $\R^{\times}$.  We therefore obtain a representative of the form $t_{\aa,x}$ for some $\aa\in \TT^{\infty}$ and $x\in \R^{\times}$.

To show that different triples from $\TT^{\infty}$ correspond to different double cosets suppose that $g=[g_{ij}]$ and $g'=[g'_{ij}]$ are elements of $B$ with $gt_{\aa,x}=t_{\bb,y}g'$ for some $\aa,\bb\in \TT^{\infty}$ and $x,y\in \R^{\times}$.  The lower triangular entries give the equations
\begin{eqnarray*}
\pi^{a_{1}}g_{22} + \pi^{a_{3}}xg_{23}
&=& \pi^{b_{1}}g_{11}' \\
\pi^{a_{2}}g_{33}
&=& \pi^{b_{3}}g_{12}'y + \pi^{b_{2}}g_{22}' \\
\pi^{a_{3}}xg_{33}
&=& \pi^{b_{3}}yg_{11}'.
\end{eqnarray*}
The third equation clearly implies that $a_{3}=b_{3}$ while the first equation gives $a_{1}\leq b_{1}$ with $a_{1}=b_{1}$ whenever $a_{1}\neq a_{3}$.  However, if $a_{1}=a_{3}$ then $b_{3}\geq b_{1}\geq a_{1} = a_{3} = b_{3}$ and again $b_{1}=a_{1}$.  Similarly, $a_{2}=b_{2}$ from the second equation so $\aa=\bb$.

We now fix an $\aa\in \TT^{\infty}$ and address the admissible range of values for $x$.  If $a_{3}=\infty$ then $\pi^{a_{3}} = 0$ and it is clear that we may take $\XX^{\aa} = \{ 1 \}$ so we will assume that $a_{3}<\infty$.  Let $x,y\in \R^{\times}$ and suppose that we are able to choose elements $g_{11},g_{22},g_{33}\in \R^{\times}$ and $g_{12}, g_{13}, g_{23}\in \R$ in such a way that the following three equations hold:
\begin{eqnarray}
g_{11}      &=& g_{22}-\pi^{a_{1}}g_{12} - \pi^{a_{3}}xg_{13}  + \pi^{a_{3}-a_{1}}xg_{23} \label{E:BBg11}\\
g_{33}      &=& g_{22}-\pi^{a_{1}}g_{12} + \pi^{a_{2}}g_{23} -  \pi^{a_{1}+a_{2}}g_{13} + \pi^{a_{3}}yg_{13} +\pi^{a_{3}-a_{2}}yg_{12} \label{E:BBg33}\\
\!\!\!\!(x-y)g_{22} &=& \left(\pi^{-a_{2}}g_{12}-\pi^{-a_{1}}g_{23} +g_{13}\right)x(\pi^{a_{1}+a_{2}}-\pi^{a_{3}}y). \label{E:BBg22}
\end{eqnarray}
Then setting
\begin{equation*}\begin{array}{rclcrcl}
g_{11}' &=& g_{11}+\pi^{a_{1}}g_{12} + \pi^{a_{3}}xg_{13} & \qquad &
g_{12}' &=& g_{12}+\pi^{a_{2}}g_{13} \\
g_{22}' &=& g_{22}-\pi^{a_{1}}g_{12} + \pi^{a_{2}}g_{23} - \pi^{a_{1}+a_{2}}g_{13} & &
g_{13}' &=& g_{13} \\
g_{33}' &=& g_{33} - \pi^{a_{2}}g_{23} +\pi^{a_{1}+a_{2}}g_{13} - \pi^{a_{3}}yg_{13} & &
g_{23}' &=& g_{23} - \pi^{a_{1}}g_{13}
\end{array}\end{equation*}
gives elements $g=[g_{ij}]$ and $g'=[g'_{ij}]$ of $B$ with
\begin{equation*}
gt_{\aa,x} = t_{\aa,y} g'.
\end{equation*}

On the other hand, given $x,y\in \R^{\times}$ we see that if $g=[g_{ij}]$ and $g'=[g'_{ij}]$ are elements of $B$ with $gt_{\aa,x} = t_{\aa,y}g'$ then (\ref{E:BBg11}\ndash\ref{E:BBg22}) hold.  Hence $t_{\aa,x}$ and $t_{\aa,y}$ represent the same double coset precisely when such solutions exist.

First suppose that $a_{1}+a_{2}\neq a_{3}$.  If $t_{\aa,x}$ and $t_{\aa,y}$ represent the same double coset for distinct $x,y\in \R^{\times}$ then from (\ref{E:BBg22}) we see that
\begin{equation} \label{E:BBvalneq}
\val(x-y) \geq \min\{a_{1},a_{2},a_{3}-a_{1},a_{3}-a_{2}\}.
\end{equation}

Conversely, suppose that we have distinct $x,y\in \R^{\times}$ so that (\ref{E:BBvalneq}) holds.
If the minimum occurs for $a_{1}$ then $a_{3}-a_{2}>a_{1}$, since $a_{1}+a_{2}\neq a_{3}$, and we have $\val(\pi^{a_{1}+a_{2}}-\pi^{a_{3}}y) = a_{1}+a_{2}$.  Setting $g_{23}$ and $g_{13}$ both to be zero and choosing $g_{12}$ with $\val(g_{12})=\val(x-y)=a_{1}$ will give $g_{22}\in \R^{\times}$ by (\ref{E:BBg22}).  Further, $g_{11},g_{33}\in \R^{\times}$ by (\ref{E:BBg11}) and (\ref{E:BBg33}) since $a_{3}-a_{2}>1$.  If the minimum occurs for $a_{3}-a_{1}$ then $a_{3}-a_{1}<a_{2}$ and $\val(\pi^{a_{1}+a_{2}}-\pi^{a_{3}}y) = a_{3}$.  Taking $g_{12}$ and $g_{13}$ to be zero and $g_{23}$ such that $\val(g_{23})=\val(x-y)-(a_{3}-a_{1})$ gives $g_{22}\in \R^{\times}$ and, provided that we make the specific choice $g_{23}=y-x$ when $a_{3}=a_{1}$, we will also have $g_{11},g_{33}\in \R^{\times}$.  The arguments when the minimum is $a_{2}$ or $a_{3}-a_{2}$ are similar and so we obtain a solution of (\ref{E:BBg11}\ndash\ref{E:BBg22}) in each case.  Hence for every $\aa\in \TT^{\infty}$ with $a_{1}+a_{2}\neq a_{3}$ we may take a representative $t_{\aa,x}$ with $x$ lying in the set
\begin{equation*}\XX^{\aa} = \R/\P^{\min\{a_{1},a_{2},a_{3}-a_{1},a_{3}-a_{2}\}}\end{equation*}
and distinct elements of this set give distinct double cosets.

Now suppose that $a_{1}+a_{2}=a_{3}$ and that $t_{\aa,x}$ and $t_{\aa,y}$ represent the same double coset for distinct elements $x,y\in \R^{\times}$.
If $\val(1-x)>\val(1-y)$ then $\val(x-y) = \val((1-y)-(1-x)) = \val(1-y)$
and if $\val(1-x)<\val(1-y)$ then $\val(x-y) = \val(1-x)<\val(1-y)$.  However, from (\ref{E:BBg22}) we know that
$\val(x-y) \geq \min\{a_{1},a_{2}\} + \val(1-y) > \val(1-y).$
Therefore, we must have
\begin{equation}\label{E:BBvaleq}
\val(1-x)=\val(1-y)=i\text{\quad and\quad}\val(x-y)\geq \min\{a_{1},a_{2}\}+i.
\end{equation}

Conversely, let $x,y\in \R^{\times}$ be such that condition (\ref{E:BBvaleq}) holds. If $a_{1}\leq a_{2}$ then choosing $g_{23}=g_{13}=0$ and $g_{12}$ with $\val(g_{12}) = \val(x-y)-a_{1}-i$ gives $g_{22}\in \R^{\times}$ and $g_{11},g_{33}\in \R^{\times}$ since $a_{3}-a_{2}=a_{1}>0$. If $a_{2}\leq a_{1}$ there is a similar argument and we again have a solution of (\ref{E:BBg11}\ndash\ref{E:BBg22}) in each case.  Hence for every $\aa\in \TT^{\infty}$ with $a_{1}+a_{2}= a_{3}$ we may take a representative $t_{\aa,x}$ with $x$ from the set
\begin{equation*}\XX^{\aa} = \bigcup_{i=0}^{\infty} (1+\pi^{i}\R^{\times}) \cap (\R/\P^{\min\{a_{1},a_{2}\} +i})^{\times}\end{equation*}
and distinct elements of this set give distinct double cosets.
\end{proof}

\begin{theorem} \label{thm:BBcosets}
A complete set of $(B,B)$-double cosets in $K$ is given by
\begin{equation*}
\RR = \{t_{\aa,x}, s_{1}^{(\alpha,\beta)},s_{2}^{(\alpha,\beta)},s_{1}s_{2}^{(\alpha)},s_{2}s_{1}^{(\alpha)},w_{0} :
\aa\in\TT^{\infty}, x\in \XX^{\aa}, \alpha,\beta\in \nn\}
\end{equation*}
where
\begin{eqnarray*}
t_{\aa,x} =
\left[\begin{array}{ccc}
1            & 0           & 0 \\
\pi^{a_{1}}  & 1           & 0 \\
\pi^{a_{3}}x & \pi^{a_{2}} & 1
\end{array}\right],\!
&\!
s_{1}^{(\alpha,\beta)} =
\left[\begin{array}{ccc}
0         & 1          & 0 \\
1         & 0          & 0 \\
\pi^\beta & \pi^\alpha & 1
\end{array}\right],\!
&\!
s_{2}^{(\alpha,\beta)} =
\left[\begin{array}{ccc}
1            & 0 & 0 \\
\pi^{\beta}  & 0 & 1 \\
\pi^{\alpha} & 1 & 0
\end{array}\right],
\\
s_{1}s_{2}^{(\alpha)} =
\left[\begin{array}{ccc}
0            & 0 & 1 \\
1            & 0 & 0 \\
\pi^{\alpha} & 1 & 0
\end{array}\right],\!
&\!
s_{2}s_{1}^{(\beta)} =
\left[\begin{array}{ccc}
0 & 1           & 0 \\
0 & \pi^{\alpha}& 1 \\
1 & 0           & 0
\end{array}\right],\!
&\!
w_{0} =
\left[\begin{array}{ccc}
0 & 0 & 1 \\
0 & 1 & 0 \\
1 & 0 & 0
\end{array}\right].
\end{eqnarray*}
\end{theorem}
\begin{proof}
We have shown that each double coset has a representative of the form $wk$ for some $w\in W$, $k\in V_{w}$ and that, in particular, when $w=1$ we can take it to be $t_{\aa,x}$ with $\aa\in \TT^{\infty}$, $x\in \XX^{\aa}$.  If $w\neq 1$ then $k$ has at most two non-zero entries below the diagonal and we are able to independently scale these by any element of $\R^{\times}$ via left and right multiplication by $T$.  This means that each representative can be chosen from the set $\RR$ described above.

Representatives associated to distinct Weyl group elements must give distinct double cosets by the Bruhat decomposition of $\mathrm{GL}(3,\mathfrak{f})$.  Further, by Proposition~\ref{prop:BBiwahori} we know that distinct elements from $\RR^{1}=\{t_{\aa,x} : \aa\in \TT^{\infty}, x\in \XX^{\aa}\}$ give distinct double cosets.  Thus we need to show that different representatives from $\RR$ with the same non-trivial Weyl group element give different double cosets.  We will prove only the case when $w=s_{1}$ and remark that the remaining cases are analogous.

If $s_{1}^{(\alpha,\beta)}$ and $s_{1}^{(\alpha',\beta')}$ represent the same double coset for some $\alpha,\alpha',\beta,\beta'\in \overline{\mathbb{Z}}_{+}$ then there must be elements $g=[g_{ij}]$ and $g'=[g'_{ij}]$ of $B$ with $gs_{1}^{(\alpha,\beta)}=s_{1}^{(\alpha',\beta')}g'$.  This implies that the following two equations hold:
\begin{eqnarray}
\label{eq:BBs1g11}
\pi^{\alpha'}g_{11} &=& \pi^{\alpha}g_{33} - \pi^{\alpha+\alpha'}g_{13} - \pi^{\alpha+\beta'}g_{23} \\
\label{eq:BBs1g22}
\pi^{\beta'}g_{22} &=& \pi^{\beta}g_{33} - \pi^{\beta+\beta'}g_{23}.
\end{eqnarray}
However (\ref{eq:BBs1g11}) implies that $\alpha=\alpha'$ and (\ref{eq:BBs1g22}) implies that $\beta=\beta'$.  Hence if the pairs $(\alpha,\beta)$ and $(\alpha',\beta')$ are distinct then $s_{1}^{(\alpha,\beta)}$ and $s_{1}^{(\alpha',\beta')}$ represent different double cosets.\end{proof}


\section{General double cosets} \label{sec:cccd}

We now turn our attention to the case of $(C_{\cc},C_{\dd})$-double cosets for $\cc,\dd\in \TT$.  In this situation it is possible for the image of $C_{\cc}$ or $C_{\dd}$ in $\mathrm{GL}(3,\mathfrak{f})$ to be a proper parabolic subgroup and so different Weyl group elements could represent the same double coset.  To eliminate these duplications we introduce the subset $W_{\cc,\dd}$ of $W$ defined as follows:
\begin{enumerate}[(i)]
\item $W_{\cc,\dd}=\{1,s_{1},s_{2},s_{1}s_{2},s_{2}s_{1},w_{0}\}$ if $\cc,\dd\succeq (1,1,1)$;
\item $W_{\cc,\dd}=\{1,s_{1},w_{0}\}$ if $\cc=(c,0,c)$ with $c>0$ and $\dd\succeq (1,1,1)$, or vice versa;
\item $W_{\cc,\dd}=\{1,s_{2},w_{0}\}$ if $\cc=(0,c,c)$ with $c>0$ and $\dd\succeq (1,1,1)$, or vice versa;
\item $W_{\cc,\dd}=\{1,w_{0}\}$ if $\cc=(c,0,c)$ or $(0,c,c)$ and $\dd=(d,0,d)$ or $(0,d,d)$;
\item $W_{\cc,\dd}=\{1\}$ if $\cc=(0,0,0)$ or $\dd=(0,0,0)$.
\end{enumerate}
Since $W_{\cc,\dd}$ forms a set of representatives for the corresponding double cosets in $\mathrm{GL}(3,\mathfrak{f})$ this ensures that representatives associated to distinct elements of $W_{\cc,\dd}$ will indeed yield distinct double cosets.  We therefore need to identify a set $\RR_{\cc,\dd}^{w}$ of representatives associated to each $w\in W_{\cc,\dd}$.  As in the previous section, we begin by looking at the set $\RR_{\cc,\dd}^{1}$ of representatives corresponding to the trivial element of $W$.

\begin{definition} \label{D:Tcd}
Define the set of triples
\begin{equation*}\TT^{1}= \{ (a_{1},a_{2},a_{3}) \in \mathbb{Z}^{3} : 1 \leq a_{1},a_{2} \leq a_{3} \}\end{equation*}
and for any $\cc,\dd\in \TT$ let
\begin{equation} \label{E:Tcddef}
\TT_{\cc,\dd} = \{ \aa \in \TT^{1} : \aa\preceq \underline{\cc},\ \aa\preceq \underline{\dd} \text{ and } a_{3} \leq \min\{ a_{1}+\underline{c}_{2}, \underline{d}_{1} + a_{2}\}\}
\end{equation}
with the following exceptions:
\begin{equation} \label{E:Tdefexcept}
\TT_{\cc,\dd} = \left\{\begin{array}{ll}
\{(1,1,1)\}                             & \text{if $\cc=(0,0,0)$ or $\dd=(0,0,0)$};  \\
\{(1,a,a):a\leq \min\{c_{2},d_{2}\}\}   & \text{if $c_{1}=d_{1}=0$ and $c_{2},d_{2}>0$;} \\
\{(a,1,a):a\leq \min\{c_{1},d_{1}\}\}   & \text{if $c_{2}=d_{2}=0$ and $c_{1},d_{1}>0$.}
\end{array}\right.
\end{equation}
Here $\underline{\cc}=(\underline{c}_{1},\underline{c}_{2},\underline{c}_{3})$ where $\underline{c}_{i} = \max \{ c_{i}, 1\}$ for each $i$.
\end{definition}

\begin{lemma} \label{lem:tccdd}
Let $\cc,\dd\in \TT$, then each $t_{\aa,x}\in \RR_{\cc,\dd}^{1}$ may be chosen with $\aa\in \TT_{\cc,\dd}$.  Moreover, if $\aa,\bb\in \TT_{\cc,\dd}$ are distinct then $t_{\aa,x}$ and $t_{\bb,y}$ represent distinct double cosets.
\end{lemma}
\begin{proof}
It is clear from Theorem~\ref{thm:BBcosets} that $\RR_{\cc,\dd}^{1}$ can be taken to be a subset of $\{t_{\aa,x} : \aa\in \TT^{1}, x\in \R^{\times}\}$.  One can show explicitly that for any $\aa\in \TT^{1}$ and $x\in \R^{\times}$ the double coset $C_{\cc}t_{\aa,x}C_{\dd}$ contains $t_{\bb,y}$ where $\bb\preceq\aa$ is defined by
\begin{eqnarray*}
b_{1} &=& \min\{a_{1},\underline{c}_{1},\underline{d}_{1} \}, \\
b_{2} &=& \min\{a_{2},\underline{c}_{2},\underline{d}_{2} \},\\
b_{3} &=& \min\{a_{3},c_{3},d_{3},a_{1}+\underline{c}_{2},\underline{d}_{1}+a_{2} \}.
\end{eqnarray*}
Thus, all double coset representatives in $\RR_{\cc,\dd}^{1}$ may be chosen with $\aa$ in the set defined by (\ref{E:Tcddef}).  When $c_{1}=d_{1}=0$ we can replace $\underline{d}_{1}+a_{2}$ by $a_{2}$ in the definition of $b_{3}$ and, similarly, when $c_{2}=d_{2}=0$ we can replace $a_{1}+\underline{c}_{2}$ by $a_{1}$.  In these exceptional cases we may therefore choose $\aa$ from one of the sets given in (\ref{E:Tdefexcept}).

We wish to show that distinct triples $\aa$ and $\bb$ from $\TT_{\cc,\dd}$ yield distinct double cosets so suppose that $g=[g_{ij}]\in C_{\cc}$ and $g'=[g'_{ij}]\in C_{\dd}$ are such that $g t_{\aa,x} = t_{\bb,y} g'$ for some $x,y\in \R^{\times}$.  Write $g_{21}=\gamma_{21}\pi^{c_{1}}$, $g_{21}'=\gamma_{21}'\pi^{d_{1}}$, $g_{32}=\gamma_{32}\pi^{c_{2}}$,  $g_{32}'=\gamma_{32}'\pi^{d_{2}}$, $g_{31}=\gamma_{31}\pi^{c_{3}}$  and  $g_{31}'=\gamma_{31}'\pi^{d_{3}}$ where $\gamma_{ij},\gamma_{ij}'\in \R$.  Comparing the lower triangular elements in the above product gives the following three equalities:
\begin{eqnarray}
\gamma_{21}\pi^{c_{1}} + g_{22}\pi^{a_{1}} + g_{23}x\pi^{a_{3}}
&=& g_{11}'\pi^{b_{1}} + \gamma_{21}'\pi^{d_{1}} \label{E:tcd1}\\
\gamma_{32}\pi^{c_{2}} + g_{33}\pi^{a_{2}}
&=& g_{12}'y\pi^{b_{3}} + g_{22}'\pi^{b_{2}} + \gamma_{32}'\pi^{d_{2}} \label{E:tcd2}\\
\gamma_{31}\pi^{c_{3}} + \gamma_{32}\pi^{c_{2}+a_{1}} + g_{33}x\pi^{a_{3}}
&=& g_{11}'y\pi^{b_{3}} + \gamma_{21}'\pi^{d_{1}+b_{2}} + \gamma_{31}'\pi^{d_{3}}. \label{E:tcd3}
\end{eqnarray}

We will assume first that $c_{1}$ and $d_{1}$ are not both zero and that $c_{2}$ and $d_{2}$ are not both zero.  In this case we see that $\val(\gamma'_{21}\pi^{d_{1}+b_{2}})\geq \underline{d}_{1}+b_{2}$, since if $d_{1}=0$ then $c_{1}>0$ forces $\val(\gamma'_{21})>0$ by (\ref{E:tcd1}), and similarly $\val(\gamma_{32}\pi^{c_{2}+a_{1}})\geq \underline{c}_{2}+a_{1}$.

If either of $a_{3}$ or $b_{3}$ is strictly less than $\min\{c_{3},d_{3},a_{1}+\underline{c}_{2},\underline{d}_{1}+a_{2}\}$ then (\ref{E:tcd3}) implies that $a_{3}=b_{3}$.  However, $a_{3}$ and $b_{3}$ cannot be greater than this minimum, since $\aa,\bb\in\TT_{\cc,\dd}$, so the only other possibility is that they are both equal to it.  Further, if either $a_{1}$ or $b_{1}$ is less than $\min\{\underline{c}_{1},\underline{d}_{1},a_{3}\}$ then (\ref{E:tcd1}) gives $a_{1}=b_{1}$, but again the only other option is for them both to be equal to this minimum.  Similarly, (\ref{E:tcd2}) shows that $a_{2}=b_{2}$ and so we have $\aa=\bb$.

Now assume that $c_{1}=d_{1}=0$ and note that this means that we may have $\gamma_{21}'$ of valuation $0$.  In this case our triples $\aa$ and $\bb$ are such that $a_{2}=a_{3}$ and $b_{2}=b_{3}$ with $a_{1}=b_{1}=1$.  If either of $a_{2}$ or $b_{2}$ is less than $\min\{c_{2},d_{2}\}$ then (\ref{E:tcd2}) implies that $a_{2}=b_{2}$.  Indeed, $a_{2}$ and $b_{2}$ cannot be greater than $\min\{c_{2},d_{2}\}$ so we see that $\aa=\bb$.  A similar argument deals with the case when $c_{2}=d_{2}=0$.
\end{proof}

\begin{definition}\label{def:tcd}
For $\aa\in \TT_{\cc,\dd}$ let
\begin{equation*}
\aa(\cc,\dd) = \min{}' \{a_{1},a_{2},a_{3}-a_{1},a_{3}-a_{2},c_{i}-a_{i},d_{i}-a_{i},a_{1}+c_{2}-a_{3},d_{1}+a_{2}-a_{3}\}
\end{equation*}
and
\begin{equation*}
\aa(\cc,\dd)' = \min{}' \{d_{3}-a_{3},c_{3}-a_{3},c_{1}-a_{1},d_{2}-a_{2}\} \geq \aa(\cc,\dd)
\end{equation*}
where $\min{}'$ means that we take $0$ if any of the terms is negative.
\end{definition}

\begin{lemma} \label{lem:rccdd1}
Let $\cc,\dd\in \TT$, then
\begin{equation*}
\RR_{\cc,\dd}^{1}  = \{ t_{\aa,x} : \aa \in \TT_{\cc,\dd}, x\in \XX_{\cc,\dd}^{\aa} \}
\end{equation*}
where
\begin{equation*}\XX_{\cc,\dd}^{\aa}  = \left\{
\begin{array}{ll}
(\R/\P^{\aa(\cc,\dd)})^{\times} & \text{if $a_{1}+a_{2}\neq a_{3}$}; \\
\bigcup_{i=0}^{\aa(\cc,\dd)'} (1+\pi^{i}\R^{\times}) \cap (\R/\P^{\aa(\cc,\dd)+i})^{\times} \cap (\R/\P^{\aa(\cc,\dd)'})^{\times} & \text{if $a_{1}+a_{2}=a_{3}$}.
\end{array}
\right.
\end{equation*}
\end{lemma}
\begin{proof}
Let $g_{ij}$, $g'_{ij}$, $\gamma_{ij}$ and $\gamma_{ij}'$ be as in the proof of Lemma~\ref{lem:tccdd}, then $gt_{\aa,x}=t_{\aa,y}g'$ for some $x,y\in \R^{\times}$ precisely when the following three equations can be solved for $g_{11}$, $g_{22}$ and $g_{33}$ in $\R^{\times}$:
\begin{eqnarray}
g_{11} \label{E:af1}
&=& g_{22} - g_{12}\pi^{a_{1}} \!- g_{13}x\pi^{a_{3}} \!+ \gamma_{21}\pi^{c_{1}-a_{1}} \!- \gamma_{21}'\pi^{d_{1}-a_{1}} \!+ g_{23}x\pi^{a_{3}-a_{1}} \\
g_{33} \label{E:af2} &=& g_{22} - g_{12}r_{y}\pi^{-a_{2}} \!- g_{13}r_{y} + g_{23}\pi^{a_{2}}\! - \gamma_{32}\pi^{c_{2}-a_{2}} \!+ \gamma_{32}'\pi^{d_{2}-a_{2}}\\
\notag\!\!\!\!\!\!\! (x-y)g_{22}
&=& \label{E:relation} (g_{12}\pi^{-a_{2}}  \!- g_{23}\pi^{-a_{1}}\!+ g_{13})xr_{y} + \gamma_{21}y\pi^{c_{1}-a_{1}} \!- \gamma_{31}\pi^{c_{3}-a_{3}} \\
& & {} - \gamma_{32}'x\pi^{d_{2}-a_{2}} \!+ \gamma_{31}'\pi^{d_{3}-a_{3}} \!- \gamma_{32}r_{x}\pi^{c_{2}-a_{2}-a_{3}} \!+ \gamma_{21}'r_{y}\pi^{d_{1}-a_{1}-a_{3}}
\end{eqnarray}
where for each $z\in \R^{\times}$ we define $r_{z} = \pi^{a_{1}+a_{2}} - z\pi^{a_{3}}.$

Suppose first that $a_{1}+a_{2}\neq a_{3}$, then (\ref{E:relation}) immediately yields
\begin{equation*}\val(x-y)\geq \aa(\cc,\dd).\end{equation*}
Conversely, given distinct elements  $x,y\in \R^{\times}$ with $\val(x-y)\geq \aa(\cc,\dd)$ then one can solve (\ref{E:relation}) for $g_{22}\in \R^{\times}$ and a careful consideration of (\ref{E:af1}) and (\ref{E:af2}) reveals that one can choose these variables so that $g_{11}$ and $g_{33}$ are invertible as well.  Thus the set
\begin{equation*}\XX_{\cc,\dd}^{\aa}= (\R/\P^{\aa(\cc,\dd)})^{\times}\end{equation*}
exactly parametrises the representatives $t_{\aa,x}$ for $\aa\in \TT_{\cc,\dd}$ with $a_{1}+a_{2}\neq a_{3}$.

Now suppose that $a_{1}+a_{2}=a_{3}$.  Let $\val(1-x)=i$ and $\val(1-y)=j$, then from (\ref{E:relation}) we see that
\begin{equation*}\val(x-y) \geq \min\{a_{1}+j,a_{2}+j,c_{2}-a_{2}+i,d_{1}-a_{1}+j,\aa(\cc,\dd)'\}.\end{equation*}
This clearly holds whenever $i,j\geq \aa(\cc,\dd)'$ since
$\val(x-y) \geq \min\{i,j\}$ so we will assume that at least one of $i$ or $j$ is less than $\aa(\cc,\dd)'$.  If $i<j$ with $i<\aa(\cc,\dd)'$ then $\val(x-y)=i$ and we must have $c_{2}=a_{2}$.  However, when $c_{2}=a_{2}$ we see that (\ref{E:relation}) gives
$\val((x-y)g_{22}+\gamma_{32}(1-x)) > i$
which implies that $\val(g_{22}-\gamma_{32})>0$, since if $\val(g_{22}-\gamma_{32})=0$ then we would have
\begin{equation*}\val((x-y)g_{22}+\gamma_{32}(1-x)) = \val((1-y)g_{22}-(1-x)(g_{22}-\gamma_{32})) = i.\end{equation*}
This in turn means that $\val(g_{33}) = \val(g_{22}-\gamma_{32})>0$ by (\ref{E:af2}) and so $g_{33}$ cannot be invertible.  Similarly, if $j<i$ with $j<\aa(\cc,\dd)'$ then $d_{1}=a_{1}$ and $g_{22}$ is not invertible by (\ref{E:af1}). Consequently, we must either have
\begin{equation} \label{eq:equalcase1}
\val(1 -  x),\ \val(1  -   y)\geq \aa(\cc,\dd)'
\end{equation}
or
\begin{equation} \label{eq:equalcase2}
\begin{array}{c}
\val(1 -  x)=\val(1 -  y)=i<\aa(\cc,\dd)' \text{ and } \\
\val(x -  y) \geq \min\{\aa(\cc,\dd) + i,\aa(\cc,\dd)'\}.
\end{array}
\end{equation}
Conversely, if $x,y\in \R^{\times}$ are distinct elements satisfying (\ref{eq:equalcase1}) or (\ref{eq:equalcase2}) then it is possible to find solutions to (\ref{E:af1}-\ref{E:relation}).
Hence, the set
\begin{equation*}\bigcup_{i=0}^{\aa(\cc,\dd)'} (1+\pi^{i}\R^{\times}) \cap (\R/\P^{\aa(\cc,\dd)+i})^{\times} \cap (\R/\P^{\aa(\cc,\dd)'})^{\times}\end{equation*}
precisely parametrises the representatives $t_{\aa,x}$ for $\aa\in \TT_{\cc,\dd}$ with $a_{1}+a_{2}=a_{3}$.
\end{proof}

Note that if $a_{1}+a_{2}=a_{3}$ then the set $\XX_{\cc,\dd}^{\aa}$ lies between $(\R/\P^{\aa(\cc,\dd)})^{\times}$ and $(\R/\P^{\aa(\cc,\dd)'})^{\times}$.  In particular, when $\aa(\cc,\dd)' = \aa(\cc,\dd)$ then the definition of $\XX_{\cc,\dd}^{\aa}$ given in Lemma \ref{lem:rccdd1} reduces to the much simpler
\begin{equation*}\XX_{\cc,\dd}^{\aa}  = (\R/\P^{\aa(\cc,\dd)})^{\times}.\end{equation*}

Further, in general we can compute directly that
\begin{equation*}
|\XX_{\cc,\dd}^{\aa}| = \left\{ \begin{array}{ll}
1                                                    & \text{if $a_{1}+a_{2}\neq a_{3}$, $\aa(\cc,\dd)=0$};\\
(q-1)q^{\aa(\cc,\dd)-1}                              & \text{if $a_{1}+a_{2}\neq a_{3}$, $\aa(\cc,\dd)>0$};\\
\aa(\cc,\dd)'+1                                      & \text{if $a_{1}+a_{2}= a_{3}$, $\aa(\cc,\dd)=0$};\\
(\aa(\cc,\dd)'-\aa(\cc,\dd)+1)(q-1)q^{\aa(\cc,\dd)-1} & \text{if $a_{1}+a_{2}= a_{3}$, $\aa(\cc,\dd)>0$}.\\
\end{array}\right.
\end{equation*}

\begin{theorem}\label{thm:cosets}
Let $\cc,\dd\in \TT$, then a complete set of $(C_{\cc},C_{\dd})$-double coset representatives in $K$ is
\begin{equation*}\RR_{\cc,\dd} = \bigcup_{w\in W_{\cc,\dd}} \RR_{\cc,\dd}^{w}\end{equation*}
where if $w\in W_{\cc,\dd}$ then we define $\RR_{\cc,\dd}^{w}$ as follows
\begin{enumerate}[{\bf (i)}]
\item $\RR_{\cc,\dd}^{1} = \{ t_{\aa,x} : \aa\in \TT_{\cc,\dd} , x\in \XX^{\aa}_{\cc,\dd}\}$;
\item $\RR_{\cc,\dd}^{s_{1}}  =  \{s_{1}^{(\alpha,\beta)}\!:\!
1\!\leq\! \alpha\! \leq \!\min\{\underline{d}_{2},c_{3}\},
1\!\leq\! \beta\!  \leq\! \min\{\underline{c}_{2},d_{3}\},
-c_{1}\! \leq \!\beta\!-\!\alpha\!\leq \!d_{1}\}$;
\item $\RR_{\cc,\dd}^{s_{2}}  = \{s_{2}^{(\alpha,\beta)} \!:\!
1\!\leq\! \alpha \!\leq \!\min\{\underline{d}_{1},c_{3}\},
1\!\leq \!\beta  \!\leq \!\min\{\underline{c}_{1},d_{3}\},
-c_{2}\! \leq \!\beta\!-\!\alpha \!\leq d_{2}\}$;
\item $\RR_{\cc,\dd}^{s_{1}s_{2}} = \{ s_{1}s_{2}^{(\alpha)} : 1\leq \alpha \leq \min\{d_{1},c_{2}\}\}$;
\item $\RR_{\cc,\dd}^{s_{2}s_{1}} = \{ s_{2}s_{1}^{(\alpha)} : 1\leq \alpha \leq \min\{c_{1},d_{2}\}\}$;
\item $\RR_{\cc,\dd}^{w_{0}}  = \{ w_{0}\}$
\end{enumerate}
and otherwise we take $\RR_{\cc,\dd}^{w} = \emptyset$.
\end{theorem}
\begin{proof}
We have already shown (i) in Lemma \ref{lem:rccdd1} so we need to consider the representatives corresponding to non-trivial Weyl group elements.  As in Theorem \ref{thm:BBcosets} we will prove the case where $w=s_{1}$ and note that the other cases are similar.

Suppose that $s_{1}^{(\alpha,\beta)}$ and $s_{1}^{(\alpha',\beta')}$ represent the same double coset for pairs $(\alpha,\beta)$ and $(\alpha',\beta')$ with $\alpha,\alpha',\beta,\beta'\geq 1$.  There must therefore be elements $g=[g_{ij}]$ of $C_{\cc}$ and $g'=[g'_{ij}]$ of $C_{\dd}$ such that $gs_{1}^{(\alpha,\beta)} = s_{1}^{(\alpha',\beta')}g'$ and, if we let $\gamma_{ij}$ and $\gamma_{ij}'$ be as in the proof of Lemma~\ref{lem:tccdd}, this occurs precisely when we can find $g_{11},g_{22},g_{33}\in \R^{\times}$ and $g_{13},g_{23},\gamma_{21},\gamma_{21}',\gamma_{32},\gamma_{32}',\gamma_{31},\gamma_{31}'\in \R$ with
\begin{eqnarray}
\label{eq:s1alphaccdd}\pi^{\alpha'}g_{11} &=&
\pi^{\alpha}g_{33} - \pi^{\alpha+\alpha'}g_{13}-\pi^{\alpha+\beta'}g_{23} - \pi^{c_{1}+\beta'}\gamma_{21} - \pi^{d_{2}}\gamma_{32}'+\pi^{c_{3}}\gamma_{31} \\
\label{eq:s1betaccdd}\pi^{\beta'}g_{22} &=&
\pi^{\beta}g_{33} - \pi^{\beta+\beta'}g_{23} - \pi^{d_{1}+\alpha'}\gamma_{21}' + \pi^{c_{2}}\gamma_{32} -\pi^{d_{3}}\gamma_{31}'.
\end{eqnarray}

First note that we must have $\val(\pi^{d_{2}}\gamma_{32}')\geq \underline{d}_{2}$ since if $d_{2}=0$ then (\ref{eq:s1alphaccdd}) forces $\val(\gamma_{32}') \geq 1$ and, similarly, $\val(\pi^{c_{2}}\gamma_{32})\geq \underline{c}_{2}$ by (\ref{eq:s1betaccdd}).  It follows that one can solve (\ref{eq:s1alphaccdd}) and (\ref{eq:s1betaccdd}) whenever $\alpha'\geq \min\{\alpha,c_{1}+\beta',\underline{d}_{2},c_{3}\}$ and $\beta'\geq \min\{\beta,d_{1}+\alpha',\underline{c}_{2},d_{3}\}$.  Thus we may choose a representative with $(\alpha,\beta)$ such that $1\leq \alpha \leq \min\{\underline{d}_{2},c_{3}\}$, $1\leq\beta\leq\min\{\underline{c}_{2},d_{3}\}$ and $-c_{1}\leq \beta-\alpha \leq d_{1}$.

Now suppose that $s_{1}^{(\alpha,\beta)}$ and $s_{1}^{(\alpha',\beta')}$ are representatives for the same double coset where $(\alpha,\beta)$ and $(\alpha',\beta')$ satisfy the restrictions above.  If either of $\alpha$ or $\alpha'$ is less than $\min\{c_{1}+\beta',\underline{d}_{2},c_{3}\}$ then $\alpha=\alpha'$ by (\ref{eq:s1alphaccdd}).  Further, if $\alpha$ and $\alpha'$ are greater than or equal to this minimum we actually have $\alpha\geq \alpha' =c_{1}+\beta'$ with $\alpha=\alpha'$ whenever $\beta=\beta'$.   Similarly, if at least one of $\beta$ or $\beta'$ is less than $\min\{d_{1}+\alpha', \underline{c}_{2},d_{3}\}$ then $\beta=\beta'$ by (\ref{eq:s1betaccdd}) and otherwise $\beta\geq \beta'=d_{1}+\alpha'$ with $\beta=\beta'$ whenever $\alpha=\alpha'$.  However, we cannot have both $\alpha'=c_{1}+\beta'$ and $\beta'=d_{1}+\alpha'$, since this would mean that $c_{1}=d_{1}=0$, so we must have $\alpha=\alpha'$ and $\beta=\beta'$.  Hence distinct pairs $(\alpha,\beta)$ give rise to distinct double cosets.
\end{proof}

\begin{remark}
The list of double coset representatives $\RR_{\cc,\dd}$ given in Theorem \ref{thm:cosets} does not seem to be symmetric in $\cc$ and $\dd$.  There is, however, a natural bijection from $C_{\cc}\backslash K/C_{\dd}$ to $C_{\dd}\backslash K /C_{\cc}$ obtained by sending each element of a double coset to its inverse.  This does indeed induce a bijection from $\RR_{\cc,\dd}$ to $\RR_{\dd,\cc}$ since we see that
\begin{equation*}\begin{array}{rcl}
(C_{\cc} t_{\aa,x} C_{\dd})^{-1} &=& C_{\dd} t_{\bb,y} C_{\cc} \\
(C_{\cc} s_{1}^{(\alpha,\beta)} C_{\dd})^{-1} &=& C_{\dd} s_{1}^{(\beta,\alpha)} C_{\cc} \\
(C_{\cc} s_{2}^{(\alpha,\beta)} C_{\dd})^{-1} &=& C_{\dd} s_{2}^{(\beta,\alpha)} C_{\cc} \\
(C_{\cc} s_{1}s_{2}^{(\alpha)} C_{\dd})^{-1} &=& C_{\dd} s_{2}s_{1}^{(\alpha)} C_{\cc} \\
(C_{\cc} s_{2}s_{1}^{(\alpha)} C_{\dd})^{-1} &=& C_{\dd} s_{1}s_{2}^{(\alpha)} C_{\cc} \\
(C_{\cc}w_{0} C_{\dd})^{-1} &=& C_{\dd} w_{0} C_{\cc}
\end{array}\end{equation*}
where
\begin{equation*}(\bb,y) = \left\{ \begin{array}{ll}
(\aa,x-\pi^{a_{3}-a_{1}-a_{2}})                         & \text{if $a_{1}+a_{2}<a_{3}$}; \\
((a_{1},a_{2},\val(r_{x})), r_{x}\pi^{-\val(r_{x})})     & \text{if $a_{1}+a_{2}=a_{3}$}; \\
((a_{1},a_{2},a_{1}+a_{2}), -1+x\pi^{a_{1}+a_{2}-a_{3}})  & \text{if $a_{1}+a_{2}>a_{3}$}.
\end{array}\right.\end{equation*}
In particular, Theorem \ref{thm:cosets} is symmetric in $\cc$ and $\dd$ with respect to this bijection.
\end{remark}

We want to use the description of the double coset structure given Theorem~\ref{thm:cosets} to investigate the components $V_{\cc}$.  From Corollary \ref{cor:altsum} we know that
\begin{equation*}
\I(V_{\cc},V_{\dd}) = \sum_{I\subseteq S_{\cc},\ J\subseteq S_{\dd}} (-1)^{|I|+|J|} |\RR_{\cc_{I},\dd_{J}}|
\end{equation*}
and for each $w\in W$ we will consider
\begin{equation*}
\I(V_{\cc},V_{\dd})^{w} = \sum_{I\subseteq S_{\cc},\ J\subseteq S_{\dd}} (-1)^{|I|+|J|} |\RR^{w}_{\cc_{I},\dd_{J}}|
\end{equation*}
since then $\I(V_{\cc},V_{\dd}) = \sum_{w\in W}\I(V_{\cc},V_{\dd})^{w}.$


\section{One descendant} \label{sec:one}

\begin{figure}[t]\fontsize{9}{9}\selectfont{
$$\begin{array}{ccccc}
{\xymatrix@C=1pt{
 \disp{\cc}{(0, c_{2}, c_{2})} \ar[dr] &             \\
                & \disp{\cc_{\{3\}}}{(0, c_{2}-1, c_{2}-1)} }}
& \quad&
{\xymatrix@C=1pt{
 \disp{\cc}{(c_{1},c_{2},c_{1}+c_{2})} \ar[d]  \\
\disp{\cc_{\{3\}}}{(c_{1},c_{2},c_{1}+c_{2}-1)}}}
& \quad&
{\xymatrix@C=1pt{
            &  \disp{\cc}{(c_{1},0,c_{1})} \ar[dl] \\
\disp{\cc_{\{3\}}}{(c_{1}-1,0,c_{1}-1)}                  }}
\end{array}$$}
\begin{caption} {The one descendant cases\label{fig:onedesc}} \end{caption}
\end{figure}

Rather than consider all possible $\cc$, we split the problem into three separate cases depending on the number of triples immediately below $\cc$ in the poset $\TT$.  Recall that the space of $K_{1}$-fixed vectors in $V$ decomposes as
\begin{equation*}V^{K_{1}} = V_{(0,0,0)} \oplus V_{(0,1,1)} \oplus V_{(1,0,1)} \oplus V_{(1,1,1)}\end{equation*}
where $(0,1,1)$ and $(1,0,1)$ are the triples having exactly one descendant in $\TT$.  The corresponding components $V_{(0,1,1)}$ and $V_{(1,0,1)}$ are equivalent irreducibles and we find that single descendant triples (see Figure~\ref{fig:onedesc}) will always give irreducible components which are equivalent to all other single descendant components lying in the same level.

\begin{theorem} \label{thm:onedesc}
Let $\cc=(c_{1},c_{2},c_{1}+c_{2})$ with $c_{1}+c_{2}>1$, then $V_{\cc}$ is irreducible and
\begin{equation*}\dim V_{\cc} = q^{2c_{1}+2c_{2}-4}(q-1)(q+1)(q^2+q+1).\end{equation*}
Moreover, $V_{\cc}\simeq V_{\dd}$ for any $\dd=(d_{1},d_{2},d_{1}+d_{2})$ with $c_{1}+c_{2}=d_{1}+d_{2}$.
\end{theorem}
\begin{proof}
For ease of notation we will assume that $c_{1}\leq d_{1}$.  If we define
$\cc'=\cc_{\{3\}}$
and $\dd'=\dd_{\{3\}}$ then we want to calculate the alternating sum
\begin{equation*}\I(V_{\cc},V_{\dd}) = |\RR_{\cc,\dd}| - |\RR_{\cc',\dd}| - |\RR_{\cc,\dd'}| + |\RR_{\cc',\dd'}|.\end{equation*}

First note that if $c_{1}$, $d_{1}$, $c_{2}$ and $d_{2}$ are all non-zero then $\RR_{\cc,\dd}^{w}$,  $\RR_{\cc',\dd}^{w}$, $\RR_{\cc,\dd'}^{w}$ and $\RR_{\cc',\dd'}^{w}$ are all equal for any non-trivial Weyl group element $w\in W$ since we are only decreasing $c_{3}$ or $d_{3}$ in Theorem~\ref{thm:cosets} and these are both greater than $\max\{c_{1},d_{1},c_{2},d_{2}\}$.  In the case where one or more of $c_{1}$, $d_{1}$, $c_{2}$ or $d_{2}$ is zero the sets are equal in pairs since $c_{3}=d_{3}>1$.  Consequently, we need only consider
\begin{equation*}
\I(V_{\cc},V_{\dd})^{1}
= |\RR_{\cc,\dd}^{1}| - |\RR_{\cc',\dd}^{1}| - |\RR_{\cc,\dd'}^{1}| + |\RR_{\cc',\dd'}^{1}|.
\end{equation*}

Now, from Definition \ref{def:tcd} we see that $\TT_{\cc',\dd}$, $\TT_{\cc,\dd'}$ and $\TT_{\cc',\dd'}$ are equal while $\TT_{\cc,\dd}$ contains the additional triple $\aa=(\underline{c}_{1},\underline{d}_{2},c_{1}+c_{2})$ with
$\aa(\cc,\dd)=0$.  If one of $c_{1}$, $d_{1}$, $c_{2}$ or $d_{2}$ is zero then
$\aa(\cc,\dd)$, $\aa(\cc',\dd)$, $\aa(\cc,\dd')$ and $\aa(\cc',\dd')$ are all zero for every $\aa\in \TT_{\cc',\dd'}$.  On the other hand, if $c_{1}$, $d_{1}$, $c_{2}$ and $d_{2}$ are all non-zero then the only way that a triple $\aa \in \TT_{\cc',\dd'}$  could have $\aa(\cc,\dd)$ strictly greater than any of $\aa(\cc',\dd)$, $\aa(\cc,\dd')$ or $\aa(\cc',\dd')$ is if the minimum occurs for $c_{3}-a_{3}$.  However, $c_{3}-a_{3}=c_{1}+c_{2}-a_{3} \geq a_{1}+c_{2}-a_{3}$ so we would need $a_{1}=c_{1}$ and the minimum would therefore have been $0$.  Thus again $\aa(\cc,\dd)$, $\aa(\cc',\dd)$, $\aa(\cc,\dd')$ and  $\aa(\cc',\dd')$ must be equal for every $\aa\in \TT_{\cc',\dd'}$.  This implies that $\RR_{\cc',\dd}^{1}$, $\RR_{\cc,\dd'}^{1}$ and $\RR_{\cc',\dd'}^{1}$ are equal while $\RR_{\cc,\dd}^{1}$ has the extra representative $t_{(\underline{c}_{1},\underline{d}_{2},c_{1}+c_{2}),1}$.  Hence
\begin{equation*}
\I(V_{\cc},V_{\dd}) =
(|\RR_{\cc',\dd'}^{1}| +1) - |\RR_{\cc',\dd'}^{1}| - |\RR_{\cc',\dd'}^{1}| + |\RR_{\cc',\dd'}^{1}| = 1.
\end{equation*}
Taking $\cc=\dd$ this shows $V_{\cc}$ is irreducible and in general it implies that $V_{\cc}$ and $V_{\dd}$ must be equivalent.
\end{proof}

 In fact, the equivalences in Theorem \ref{thm:onedesc} are the only ones that can involve a component $V_{\cc}$ with a triple of the form $\cc=(c_{1},c_{2},c_{1}+c_{2})$.

\begin{proposition} \label{prop:onemult}
Let $\cc=(c_{1},c_{2},c_{1}+c_{2})$ with $c_{1}+c_{2}>1$, then the multiplicity of $V_{\cc}$ in $\Res_{K}^{\mathbb{G}(F)}V$ is $c_{1}+c_{2}+1$.
\end{proposition}
\begin{proof}
By Theorem~\ref{thm:onedesc} we may take $\cc = (0,c,c)$ where $c=c_{1}+c_{2}$.  Further, since $K_{c}$ is the largest principal congruence subgroup contained in $C_{(0,c,c)}$, we know that every subrepresentation of $\Res_{K}^{\mathbb{G}(F)}V$ equivalent to $V_{\cc}$ must be a subrepresentation of $V^{K_{c}} \simeq U_{(c,c,c)}$.  In particular, this means that the multiplicity of $V_{\cc}$ in $\Res_{K}^{\mathbb{G}(F)}V$ is equal to its multiplicity in $U_{(c,c,c)}$.  Thus, setting $\cc'=\cc_{\{3\}} = (0,c-1,c-1)$ and $\dd=(c,c,c)$ we would like to calculate
\begin{equation*}\I(V_{\cc},U_{\dd}) = |\RR_{\cc,\dd}| - |\RR_{\cc',\dd}|.\end{equation*}
Now, $\RR_{\cc,\dd}^{w_{0}} = \RR_{\cc',\dd}^{w_{0}} = \{ w_{0} \}$ and $\RR_{\cc,\dd}^{w} = \RR_{\cc',\dd}^{w} = \emptyset$ for $w=s_{1}$, $s_{1}s_{2}$ or $s_{2}s_{1}$.  However, for $w=s_{2}$ we obtain $\RR_{\cc,\dd}^{s_{2}} = \RR_{\cc',\dd}^{s_{2}} \cup \{ s_{2}^{(c,1)} \}$.   Further, for every $\aa\in \TT_{\cc',\dd}$ we have
$\aa(\cc,\dd) = \aa(\cc',\dd) = 0$ and
$\TT_{\cc,\dd} = \TT_{\cc',\dd} \cup \{ (1,a,c) : 1\leq a \leq c \}$.  Hence
$\RR_{\cc,\dd} = \RR_{\cc',\dd} \cup \{ t_{(1,a,c),1} : 1\leq a \leq c\} \cup \{s_{2}^{(c,1)} \}$
and
\begin{equation*}\I(V_{\cc},U_{\dd}) = (|\RR_{\cc',\dd}| + (c+1) ) - |\RR_{\cc',\dd}| = c+1\end{equation*}
as required.
\end{proof}

\section{Two descendants} \label{sec:two}

\begin{figure}[t]\fontsize{9}{9}\selectfont{
$${\xymatrix@C=1pt{
                          &\disp{\cc}{(c_{1},c_{2},\max\{c_{1},c_{2}\})} \ar[dr]\ar[dl] &                           \\
\disp{\cc_{\{1\}}}{(c_{1}-1,c_{2},\max\{c_{1},c_{2}\})} \ar[dr] &                                & \disp{\cc_{\{2\}}}{(c_{1},c_{2}-1,\max\{c_{1},c_{2}\})} \ar[dl] \\
                          &\disp{\cc_{\{1,2\}}}{(c_{1}-1,c_{2}-1,\max\{c_{1},c_{2}\})}           &                            }}$$}
\begin{caption} {The general two descendant case with $c_{1},c_{2}>1$ \label{fig:twogen}}\end{caption}
\end{figure}

\begin{figure}[t]\fontsize{9}{9}\selectfont{
$$\begin{array}{ccc}\xymatrix@C=1pt{
                 &\disp{\cc}{(1,c_{2},c_{2})} \ar[dl]\ar[dr] &                     \\
\disp{\cc_{\{1\}}}{(0,c_{2},c_{2})} \ar[ddr] &       & \disp{\cc_{\{2\}}}{(1,c_{2}-1,c_{2})}  \ar@{-->}[d] \\
&& \makebox[0.75\width]{$(1,c_{2}-1,c_{2}-1)$} \ar@{-->}[dl] \\
                 &\disp{\cc_{\{1,2\}}}{(0,c_{2}-1,c_{2}-1)}         &                      }
& \quad&
\xymatrix@C=1pt{
                 &\disp{\cc}{(c_{1},1,c_{1})}\ar[dl]\ar[dr]&                  \\
\disp{\cc_{\{1\}}}{(c_{1}-1,1,c_{1})}\ar@{-->}[d] &       & \disp{\cc_{\{2\}}}{(c_{1},0,c_{1})} \ar[ddl] \\
\makebox[0.75\width]{$(c_{1}-1,1,c_{1}-1)$} \ar@{-->}[dr]\\
                               & \disp{\cc_{\{1,2\}}}{(c_{1}-1,0,c_{1}-1)}            &                      }
\end{array}$$}
\begin{caption} {The extremal two descendant cases\label{fig:twoother}}\end{caption}
\end{figure}

In the decomposition of $V^{K_{1}}$ the only component corresponding to a triple with exactly two descendants in $\TT$ is $V_{(1,1,1)}$.  This is the pull-back to $K$ of the Steinberg representation of $\mathrm{GL}(3,\mathfrak{f})$ so is irreducible and appears with multiplicity $1$.  Indeed, any triple $\cc\in \TT$ with two descendants (see Figure~\ref{fig:twogen}) will give an irreducible component $V_{\cc}$ which has multiplicity $1$ in the restriction of $V$ to $K$.  Here we note that if $c_{1}=1$ then $\cc_{\{2\}}=(1,c_{2}-1,c_{2})$ but $\cc_{\{1,2\}} = (0,c_{2}-1,c_{2}-1)$ so for the purposes of calculating $\I(V_{\cc},V_{\cc})$ we ignore the triple in $\TT$ that lies between them (see Figure~\ref{fig:twoother}).  Similarly, we ignore the triple between $\cc_{\{1\}}$ and $\cc_{\{1,2\}}$ when $c_{2}=1$.

\begin{theorem}
Let $\cc=(c_{1},c_{2},\max\{c_{1},c_{2}\})$ where $c_{1},c_{2}\geq 1$ and $c_{1}+c_{2}>1$, then $V_{\cc}$ is irreducible of dimension
\begin{equation*}
\dim V_{\cc} = q^{c_{1}+c_{2}+\max\{c_{1},c_{2}\}-5}(q-1)^{2}(q+1)(q^2+q+1).
\end{equation*}
Moreover, the multiplicity of $V_{\cc}$ in $\Res_{K}^{\mathbb{G}(F)}V$ is $1$.
\end{theorem}
\begin{proof}
We will assume that $c_{1}\leq c_{2}$ so that $\cc=(c_{1},c_{2},c_{2})$ and remark that the proof for the case where $c_{1}\geq c_{2}$ is similar.  Let $\dd=(c_{2},c_{2},c_{2})$, then as in Proposition~\ref{prop:onemult} it suffices to calculate
\begin{equation*}\I(V_{\cc},U_{\dd}) = |\RR_{\cc_{\emptyset},\dd}| - |\RR_{\cc_{\{1\}},\dd}| - |\RR_{\cc_{\{2\}},\dd}| + |\RR_{\cc_{\{1,2\}},\dd}|.\end{equation*}

First suppose that $c_{1}>1$.
We begin by examining the double cosets corresponding to non-trivial $w\in W$.  Let $s_{1}^{(\alpha,\beta)}$ be a representative which belongs to $\RR_{\cc_{\emptyset},\dd}^{s_{1}}$ but not to $\RR_{\cc_{\{2\}},\dd}$.  Then $\beta=c_{2}$, since we are decreasing $c_{2}$,  and $1\leq \alpha\leq c_{2}$, since the restriction $-c_{1}\leq c_{2}-\alpha\leq c_{2}$ does not play a role.  Thus we see that $\RR_{\cc_{\emptyset},\dd}^{s_{1}} = \RR_{\cc_{\{2\}},\dd}^{s_{1}} \cup \{s_{1}^{(\alpha,c_{2})}: 1\leq \alpha \leq c_{2}\}$ and
$\RR_{\cc_{\{1\}},\dd}^{s_{1}} = \RR_{\cc_{\{1,2\}},\dd}^{s_{1}} \cup \{s_{1}^{(\alpha,c_{2})}: 1\leq \alpha \leq c_{2}\}$ by the same argument.  This means that the contribution of these representatives to the alternating sum is
\begin{equation*}\I(V_{\cc},U_{\dd})^{s_{1}} = (|\RR_{\cc_{\{2\}},\dd}^{s_{1}}|+c_{2}) - (|\RR_{\cc_{\{1,2\}},\dd}^{s_{1}}|+c_{2}) - |\RR_{\cc_{\{2\}},\dd}^{s_{1}}| + |\RR_{\cc_{\{1,2\}},\dd}^{s_{1}}| = 0.\end{equation*}

Similarly, let $s_{2}^{(\alpha,\beta)}$ be a representative lying in $\RR_{\cc_{\emptyset},\dd}^{s_{2}}$ but not in $\RR_{\cc_{\{1\}},\dd}^{s_{2}}$.  We are now decreasing $c_{1}$ so $\beta=c_{1}$ and $1\leq \alpha\leq c_{2}$ since the restriction $-c_{2}\leq c_{1}-\alpha\leq c_{2}$ is again irrelevant.  We therefore obtain $\RR_{\cc_{\emptyset},\dd}^{s_{2}} = \RR_{\cc_{\{1\}},\dd}^{s_{2}} \cup \{s_{2}^{(\alpha,c_{1})}: 1\leq \alpha \leq c_{2}\}$ and $\RR_{\cc_{\{2\}},\dd}^{s_{2}} = \RR_{\cc_{\{1,2\}},\dd}^{s_{2}} \cup \{s_{2}^{(\alpha,c_{1})}: 1\leq \alpha \leq c_{2}\}$ in the same manner giving
\begin{equation*}\I(V_{\cc},U_{\dd})^{s_{2}} = (|\RR_{\cc_{\{1\}},\dd}^{s_{2}}|+c_{2}) - |\RR_{\cc_{\{1\}},\dd}^{s_{2}}| - (|\RR_{\cc_{\{1,2\}},\dd}^{s_{2}}|+c_{2}) + |\RR_{\cc_{\{1,2\}},\dd}^{s_{2}}| = 0.\end{equation*}
Further, it is easy to check that we also have $\I(V_{\cc},U_{\dd})^{w} = 0$ for $w=s_{1}s_{2}$, $s_{2}s_{1}$ and $w_{0}$.  Hence, we only need to consider
\begin{equation*}\I(V_{\cc},U_{\dd})^{1} = |\RR^{1}_{\cc_{\emptyset},\dd}| - |\RR^{1}_{\cc_{\{1\}},\dd}|- |\RR^{1}_{\cc_{\{2\}},\dd}|+ |\RR^{1}_{\cc_{\{1,2\}},\dd}|.\end{equation*}

Now,
$\TT_{\cc_{\emptyset},\dd} = \TT_{\cc_{\{2\}},\dd} \cup \{ (a,c_{2},c_{2}) : 1\leq a \leq c_{1}\}$
and for each $\aa\in \TT_{\cc_{\{2\}},\dd}$ the only way that we can have $\aa(\cc_{\emptyset},\dd)>\aa(\cc_{\{2\}},\dd)$ is if the minimum occurs for $c_{2}-a_{2}$.  However, since $c_{2}-a_{2} = c_{3}-a_{2} \geq c_{3}-a_{3}$, this would imply that $a_{3}=a_{2}$ and so we would in fact have  $\aa(\cc_{\emptyset},\dd)=\aa(\cc_{\{2\}},\dd)=0$.  Consequently $\aa(\cc_{\emptyset},\dd)=\aa(\cc_{\{2\}},\dd)$ for every $\aa\in \TT_{\cc_{\{2\}},\dd}$ and
$\RR^{1}_{\cc_{\emptyset},\dd} = \RR^{1}_{\cc_{\{2\}},\dd} \cup \{ t_{(a,c_{2},c_{2}),1}: 1\leq a \leq c_{1}\}.$
Similarly,
$\TT_{\cc_{\{1\}},\dd} = \TT_{\cc_{\{1,2\}},\dd} \cup \{ (a,c_{2},c_{2}) : 1 \leq a \leq c_{1}-1\}$
and $\aa(\cc_{\{1\}},\dd)=\aa(\cc_{\{1,2\}},\dd)$ for every $\aa\in \TT_{\cc_{\{1,2\}},\dd}$, implying that $\RR^{1}_{\cc_{\{1\}},\dd} = \RR^{1}_{\cc_{\{1,2\}},\dd} \cup \{ t_{(a,c_{2},c_{2}),1}:1\leq a \leq c_{1}-1\}$.  Hence we obtain
\begin{equation*}\I(V_{\cc},U_{\dd}) = (|\RR^{1}_{\cc_{\{2\}},\dd}|+c_{1}) - (|\RR^{1}_{\cc_{\{1,2\}},\dd}|+(c_{1}-1))- |\RR^{1}_{\cc_{\{2\}},\dd}|+ |\RR^{1}_{\cc_{\{1,2\}},\dd}|  = 1.\end{equation*}
and $V_{\cc}$ is an irreducible subrepresentation of $U_{\dd}$ with multiplicity $1$.

Suppose now that $c_{1}=1$ so that $\cc=(1,c_{2},c_{2})$ and $c_{\{1,2\}}= (0,c_{2}-1,c_{2}-1)$.  While we still have
$\RR_{\cc_{\emptyset},\dd}^{s_{1}} = \RR_{\cc_{\{2\}},\dd}^{s_{1}} \cup \{ s_{1}^{(\alpha,c_{2})} : 1\leq \alpha \leq c_{2}\}$ it transpires that $\RR_{\cc_{\{1\}},\dd}^{s_{1}} = \RR_{\cc_{\{1,2\}},\dd}^{s_{1}} = \emptyset$ since $s_{1}$ does not belong to $W_{\cc_{\{1\}},\dd}$ or $W_{\cc_{\{1,2\}},\dd}$.  Thus in this case the contribution from these representatives becomes
\begin{equation*}\I(V_{\cc},U_{\dd})^{s_{1}} = (|\RR_{\cc_{\{2\}},\dd}^{s_{1}}|+c_{2}) - 0 - |\RR_{\cc_{\{2\}},\dd}^{s_{1}}| + 0 = c_{2}.\end{equation*}
In contrast, reducing $c_{1}$ no longer changes the inequalities in Theorem \ref{thm:cosets}(iii)  so
$\RR_{\cc_{\emptyset},\dd}^{s_{2}} = \RR_{\cc_{\{1\}},\dd}^{s_{2}}$
and
$\RR_{\cc_{\{2\}},\dd}^{s_{2}} = \RR_{\cc_{\{1,2\}},\dd}^{s_{2}} \cup \{s_{2}^{(\alpha,c_{2})}\}$ since in $\cc_{\{1,2\}}$ we also decrease the third entry.  Consequently,
\begin{equation*}\I(V_{\cc},U_{\dd})^{s_{2}} = |\RR_{\cc_{\{1\}},\dd}^{s_{2}}| - |\RR_{\cc_{\{1\}},\dd}^{s_{2}}| - (|\RR_{\cc_{\{1,2\}},\dd}^{s_{2}}| +1)+ |\RR_{\cc_{\{1,2\}},\dd}^{s_{2}}| = -1.\end{equation*}
Further, $\RR_{\cc_{\emptyset},\dd}^{s_{1}s_{2}} = \RR_{\cc_{\{2\}},\dd}^{s_{1}s_{2}} \cup \{s_{1}s_{2}^{(c_{2})}\}$ and $\RR_{\cc_{\{1\}},\dd}^{s_{1}s_{2}} = \RR_{\cc_{\{1,2\}},\dd}^{s_{1}s_{2}} = \emptyset$ giving
\begin{equation*}\I(V_{\cc},U_{\dd})^{s_{1}s_{2}} = (|\RR_{\cc_{\{2\}},\dd}^{s_{1}s_{2}}|+1) - 0 - |\RR_{\cc_{\{2\}},\dd}^{s_{1}s_{2}}| + 0 = 1\end{equation*}
while $\I(V_{\cc},U_{\dd})^{w} = 0$ for $w=s_{2}s_{1}$ and $w_{0}$.

Now, $\TT_{\cc_{\emptyset},\dd} = \TT_{\cc_{\{1\}},\dd}$ with $\aa(\cc_{\emptyset},\dd) = \aa(\cc_{\{1\}},\dd) = 0$ for all $\aa\in \TT_{\cc_{\{1\}},\dd}$ since we will always have $c_{1}=a_{1}$.  Thus $\RR_{\cc_{\emptyset},\dd}^{1}$ and  $\RR_{\cc_{\{1\}},\dd}^{1}$ are equal.  However, since in $\cc_{\{1,2\}}$ we also reduce the $c_{3}$ entry, we have
$\TT_{\cc_{\{2\}},\dd} = \TT_{\cc_{\{1,2\}},\dd} \cup \{ (1,a_{2},c_{2}): 1\leq a_{2} \leq c_{2}-1\}$.  Again $\aa(\cc_{\{2\}},\dd) = \aa(\cc_{\{1,2\}}, \dd) =0$ for every $\aa\in \TT_{\cc_{\{1,2\}}, \dd}$ so
$\RR_{\cc_{\{1\}},\dd}^{1} = \RR_{\cc_{\{1,2\}},\dd}^{1} \cup \{t_{(1,a_{2},c_{2}),1} : 1\leq a_{2} \leq c_{2}-1\}$.  Thus
\begin{equation*}\I(V_{\cc},U_{\dd})^{1} = |\RR_{\cc_{\{1\}},\dd}^{1}| - |\RR_{\cc_{\{1\}},\dd}^{1}| - (|\RR_{\cc_{\{1,2\}},\dd}^{1}| + (c_{2}-1)) + |\RR_{\cc_{\{1,2\}},\dd}^{1}| = -(c_{2}-1).\end{equation*}
Hence, overall we obtain
\begin{equation*}\I(V_{\cc},U_{\dd}) = -(c_{2}-1) + c_{2} + (-1) + 1 + 0 + 0 = 1\end{equation*}
and $V_{\cc}$ is an irreducible subrepresentation of $U_{\dd}$ with multiplicity $1$.
\end{proof}


\section{Three descendants} \label{sec:three}

\begin{figure}[t]\fontsize{9}{9}\selectfont{
$$\xymatrix@C=1pt{
& \disq{\cc}{(c_{1},c_{2},c_{1}+c_{2}-k)} \ar[dl]\ar[d]\ar[dr] & \\
\disq{\cc_{\{1\}}}{(c_{1}-1,c_{2},c_{1}+c_{2}-k)} \ar[d]\ar[dr] &
\disq{\cc_{\{3\}}}{(c_{1},c_{2},c_{1}+c_{2}-k-1)} \ar[dl]\ar[dr] &
\disq{\cc_{\{2\}}}{(c_{1},c_{2}-1,c_{1}+c_{2}-k)} \ar[dl]\ar[d]\\
\disq{\cc_{\{1,3 \}}}{(c_{1}-1,c_{2},c_{1}+c_{2}-k-1)} \ar[dr] &
\disq{\cc_{\{1,2\}}}{(c_{1}-1,c_{2}-1,c_{1}+c_{2}-k)} \ar[d] &
\disq{\cc_{\{2,3\}}}{(c_{1},c_{2}-1,c_{1}+c_{2}-k-1)}\ar[dl] \\
& \disq{\cc_{\{1,2,3\}}}{(c_{1}-1,c_{2}-1,c_{1}+c_{2}-k-1)}}$$}
\begin{caption} {The three descendant case with $1<k<\min\{c_{1},c_{2}\}$\label{fig:threegen}} \end{caption}
\end{figure}

\begin{figure}[t]\fontsize{9}{9}\selectfont{
$$\xymatrix@C=1pt{
& \disq{\cc}{(c_{1},c_{2},c_{1}+c_{2}-1)} \ar[dl]\ar[d]\ar[dr] & \\
\disq{\cc_{\{1\}}}{(c_{1}-1,c_{2},c_{1}+c_{2}-1)} \ar[d] &
\disq{\cc_{\{3\}}}{(c_{1},c_{2},c_{1}+c_{2}-2)} \ar[dl]\ar[dr] &
\disq{\cc_{\{2\}}}{(c_{1},c_{2}-1,c_{1}+c_{2}-1)} \ar[d]\\
\disq{\cc_{\{1,3 \}}}{(c_{1}-1,c_{2},c_{1}+c_{2}-2)} \ar@{-->}[dr] & &
\disq{\cc_{\{2,3\}}}{(c_{1},c_{2}-1,c_{1}+c_{2}-2)}\ar@{-->}[dl] \\
& \disq{\cc_{\{1,2\}}=\cc_{\{1,2,3\}}}{(c_{1}-1,c_{2}-1,c_{1}+c_{2}-2)}}$$}
\begin{caption} {The three descendant case with $k=1$\label{fig:threeother}} \end{caption}
\end{figure}

The remaining case, where $\cc$ has three triples immediately beneath it in $\TT$ (see Figure \ref{fig:threegen}), does not appear in the decomposition of $V^{K_{1}}$ and we find that these components are reducible in general.  Consider $\cc=(c_{1},c_{2},c_{3})$ as part of a chain of triples
\begin{equation*}(c_{1},c_{2},c_{1}+c_{2}) \succeq \cdots \succeq \cc \succeq \cdots \succeq (c_{1},c_{2},\max\{c_{1},c_{2}\}).\end{equation*}
We let $k$ denote the position of $\cc$ in this chain, so that $c_{3}=c_{1}+c_{2}-k$, and $\ell=\min\{c_{1},c_{2}\}$ the length of the chain.  The number of intertwining operators $\I(V_{\cc},V_{\cc})$ turns out to be a polynomial in $q$ whose degree is the minimum of $k$ and $\ell-k$.  Further, two triples correspond to equivalent components precisely when their chains start at the same level $c_{1}+c_{2}$, they have the same position $k$ in their chain and that position is in the first half of the chain.  Here we note that when $k=1$ the triples $\cc_{\{1,3\}}$ and $\cc_{\{1,2,3\}}$ will be equal (see Figure \ref{fig:threeother}) so their contributions will cancel in the alternating sum for $V_{\cc}$.

\begin{theorem} \label{thm:threedesc}
Let $\cc=(c_{1},c_{2},c_{1}+c_{2}-k)$ with $0<k<\ell=\min\{c_{1},c_{2}\}$, then
\begin{equation*}\I(V_{\cc},V_{\cc}) = \left\{ \begin{array}{ll}
q-2                 & \text{if $k=1$;}\\
(q-1)^2q^{k-2}      & \text{if $1< k \leq \lfloor\ell/2\rfloor$;}\\
(q-1)q^{\ell-k-1}   & \text{if $\lfloor\ell/2\rfloor < k < \ell-1$;}\\
(q-1)               & \text{if $k=\ell-1$.}
\end{array}
\right.\end{equation*}
Moreover, let $\dd=(d_{1},d_{2},d_{1}+d_{2}-k')$ with $0<k'<\ell'=\min\{d_{1},d_{2}\}$.  If
we have
\begin{enumerate}[(i)]
\item $c_{3}=d_{3}$;
\item $k=k'$; and
\item $k\leq \lfloor \min\{\ell,\ell'\}/2 \rfloor$
\end{enumerate}
then $V_{\cc}\simeq V_{\dd}$, otherwise $\I(V_{\cc},V_{\dd})=0$.
\end{theorem}

We will prove Theorem \ref{thm:threedesc} in a series of steps.  Let $\cc$ and $\dd$ be as above.  When $c_{3}\neq d_{3}$ it is clear that we will have $\I(V_{\cc},V_{\dd})=0$ so we can assume that $c_{3}=d_{3}$.  In particular, this means that $d_{1},d_{2}<c_{3}$ and $c_{1},c_{2}<d_{3}$ so if $w\neq 1$ then we have $\RR^{w}_{\cc_{I},\dd_{J}} = \RR^{w}_{\cc_{I},\dd_{J\cup\{3\}}}$ for each $I,J\subseteq S=\{1,2,3\}$.  The contribution from the double cosets supported on non-trivial Weyl group elements is therefore $0$ and
\begin{equation*}\I(V_{\cc},V_{\dd}) = \sum_{I,J\subseteq S} (-1)^{|I|+|J|} | \RR^{1}_{\cc_{I},\dd_{J}}|.\end{equation*}
Consequently, for a fixed triple $\aa\in \TT_{\cc,\dd}$ we will consider the alternating sum
\begin{equation*}\I_{\aa} = \sum_{I,J,\subseteq S}(-1)^{|I|+|J|} |\XX^{\aa}_{\cc_{I},\dd_{J}}|\end{equation*}
where we take $|\XX^{\aa}_{\cc_{I},\dd_{J}}|=0$ for $\aa\notin \TT_{\cc_{I},\dd_{J}}$.  This gives
$\I(V_{\cc},V_{\dd}) = \sum_{\aa\in \TT_{\cc,\dd}} \I_{\aa}.$

We begin by showing that $V_{\cc}$ and $V_{\dd}$ will have no constituents in common if conditions (i\ndash iii) in the Theorem are not met.

\begin{lemma} \label{lem:kequal}
If $k\neq k'$, then $\I_{\aa}=0$ for every $\aa\in \TT_{\cc,\dd}$.
\end{lemma}
\begin{proof}
Assume to the contrary that $\aa$ is a triple in $\TT_{\cc,\dd}$ with $\I_{\aa}\neq 0$.  We begin by showing that this cannot happen in the case where $a_{1}+a_{2}\neq a_{3}$.

Suppose that $c_{3}-a_{3}>\aa(\cc,\dd)$.  For each $I,J\subseteq S$ we see that $\aa$ belongs to $\TT_{\cc_{I},\dd_{J}}$ precisely when it belongs to $\TT_{\cc_{I\cup \{3\}}, \dd_{J}}$.  Moreover, if $\aa\in\TT_{\cc_{I},\dd_{J}}$ then $\aa(\cc_{I},\dd_{J}) = \aa(\cc_{I\cup \{3\}},\dd_{J})$ since decreasing $c_{3}$ by $1$ does not change the minimum.  However, this means that $\XX^{\aa}_{\cc_{I},\dd_{J}} = \XX^{\aa}_{\cc_{I\cup \{3\}},\dd_{J}}$ and we actually have
\begin{equation*}\I_{\aa} = \sum_{J\subseteq S} \sum_{I\subseteq \{1,2\}} (-1)^{|I|+|J|}\left(|\XX^{\aa}_{\cc_{I},\dd_{J}}| - |\XX^{\aa}_{\cc_{I\cup \{3\}},\dd_{J}}|\right) = 0.\end{equation*}

Now suppose that $c_{3}-a_{3}=\aa(\cc,\dd)$ but that $c_{1}-a_{1}>\aa(\cc,\dd)$.  When $k>1$ the same approach can be used to show that $\I_{\aa}=0$ so we need only consider the $k=1$ case.  Then $c_{3}=c_{1}+c_{2}-1$ gives $a_{1}+c_{2}-a_{3}=(c_{3}-a_{3})-(c_{1}-a_{1})+1 < 1$ implying that $\aa(\cc,\dd)=a_{1}+c_{2}-a_{3}=0$.  In particular, $\aa\notin \TT_{\cc_{I},\dd_{J}}$ if $2\in I$ and $\aa(\cc_{I},\dd_{J})=\aa(\cc_{I\cup\{1\}},\dd_{J})$ otherwise.  This again means that
\begin{equation*}\I_{\aa} = \sum_{J\subseteq S} \sum_{I\subseteq \{3 \}}(-1)^{|I|+|J|}\left(|\XX^{\aa}_{\cc_{I},\dd_{J}}| - |\XX^{\aa}_{\cc_{I\cup \{1\}},\dd_{J}}|\right) = 0.\end{equation*}

Similarly, $\I_{\aa}=0$ if $d_{2}-a_{2}>\aa(\cc,\dd)$ so the only triples $\aa$ that could correspond to non-zero $\I_{\aa}$ are those with $c_{1}-a_{1}=d_{2}-a_{2}=c_{3}-a_{3}=\aa(\cc,\dd)$.  Note that in this case $a_{1}+c_{2}-a_{3}=k$ and $d_{1}+a_{2}-a_{3}=k'$ with $c_{1}\leq d_{1}$ and $c_{2}\geq d_{2}$.  If $k<k'$ then we must have $c_{1}<d_{1}$.  Consequently, $d_{1}-a_{1}$ and $d_{1}+a_{2}-a_{3}$ are both greater than $\aa(\cc,\dd)$ and we can show that $\I_{\aa}=0$ since $k'>1$.   On the other hand, if $k>k'$ then $c_{2}>d_{2}$ implies that $c_{2}-a_{2}$ and $a_{1}+c_{2}-a_{3}$ are greater than $\aa(\cc,\dd)$ and $\I_{\aa}=0$.  Hence, when $k\neq k'$ there cannot be a triple $\aa$ with $a_{1}+a_{2}\neq a_{3}$ which has non-zero $\I_{\aa}$.

We now consider the case where $a_{1}+a_{2}=a_{3}$ and note that $\aa(\cc,\dd)$ reduces to the minimum of $a_{1}$, $a_{2}$, $c_{2}-a_{2}$, $d_{1}-a_{1}$ and $\aa(\cc,\dd)'$.  By an argument essentially identical to that given above we see that $\I_{\aa}$ can only be non-zero for triples $\aa$ with $c_{1}-a_{1}=d_{2}-a_{2}=c_{3}-a_{3}=\aa(\cc,\dd)'$ and these have $c_{2}-a_{2}=k$ and $d_{1}-a_{1}=k'$.  If $k<k'$ then $d_{1}-a_{1}>\aa(\cc,\dd)$ implying that $\I_{\aa}=0$ whereas $k>k'$ gives $c_{2}-a_{2}>\aa(\cc,\dd)$ and again $\I_{\aa}=0$.  Hence, we again see that when $k\neq k'$ no triples $\aa$ with $a_{1}+a_{2}= a_{3}$ have non-zero $\I_{\aa}$.
\end{proof}

\begin{lemma}
If $\cc\neq\dd$ but $k=k'$, then
\begin{equation*}\I(V_{\cc},V_{\dd}) = \left\{\begin{array}{ll}
\I_{(c_{1}-k,d_{2}-k,c_{3}-k)} & \text{if $c_{1}<d_{1}$ and $k\leq \lfloor \min\{\ell,\ell'\}/2\rfloor$}; \\
\I_{(d_{1}-k,c_{2}-k,d_{1}+c_{2}-2k)} & \text{if $c_{1}>d_{1}$ and $k\leq \lfloor \min\{\ell,\ell'\}/2\rfloor$}; \\
0 & \text{otherwise}.
\end{array}\right.\end{equation*}
\end{lemma}

\begin{proof}
Assume that $\aa\in \TT_{\cc,\dd}$ has $\I_{\aa}\neq 0$.  When $a_{1}+a_{2}\neq a_{3}$ the proof of Lemma~\ref{lem:kequal} tells us that $\aa=(c_{1}-i,d_{2}-i,d_{3}-i)$ where $i=\aa(\cc,\dd)$ and, moreover, that this can only happen if $c_{1}< d_{1}$ and $c_{2}> d_{2}$.  If $k> \aa(\cc,\dd)$ then $d_{1}-a_{1}$ and $d_{1}+a_{2}-a_{3}=k$ are both greater than $\aa(\cc,\dd)$ which implies that $\I_{\aa}=0$.  Thus we must have $i=k$ and $k\leq \min\{a_{1},a_{2},a_{3}-a_{1},a_{3}-a_{2}\}=\min\{\ell,\ell'\}-k$ gives $k\leq \lfloor \min\{\ell,\ell'\}/2\rfloor$.

Similarly, when $a_{1}+a_{2}=a_{3}$ we know that $\aa=(d_{1}-k,c_{2}-k,d_{1}+c_{2}-2k)$ where $k=\aa(\cc,\dd)$ and that this only happens for $c_{1}>d_{1}$ and $c_{2}<d_{2}$.  Further, $k\leq \min\{a_{1},a_{2}\}=\min\{\ell,\ell'\}-k$ again implies that $k\leq \lfloor \min\{\ell,\ell'\}/2\rfloor$.
\end{proof}

\begin{lemma} \label{lem:iasum}
For any $\cc$,
\begin{equation*}\I(V_{\cc},V_{\cc}) = \sum_{i=0}^{\min\{k,\ell-k\}} \I_{(c_{1}-i,c_{2}-i,c_{3}-i)}.\end{equation*}
\end{lemma}
\begin{proof}
As in the previous Lemma, we know that only triples $\aa\in \TT_{\cc,\cc}$ of the form $\aa=(c_{1}-i,c_{2}-i,c_{3}-i)$ with $i=\aa(\cc,\cc)$ can possibly contribute to $\I(V_{\cc},V_{\cc})$.   However, in this case our only restrictions are that $i\leq a_{1}+c_{2}-a_{3}=k$ and $i\leq \min\{a_{3}-a_{1},a_{3}-a_{2}\} = \ell-k$ so we need to include all such $\aa$ with $0\leq i \leq \min\{k,\ell-k\}$.
\end{proof}

\begin{lemma}\label{lem:ia}
Let $\cc=\dd$ and $i\leq\min\{k,\ell-k\}$.  For $i\neq k$ we have
\begin{equation*}\I_{(c_{1}-i,c_{2}-i,c_{3}-i)} = \left\{\begin{array}{ll}
1               & \text{if $i=0$;}\\
q-2             & \text{if $i=1$;}\\
(q-1)^2q^{i-2}  & \text{if $i>1$}
\end{array}
\right.\end{equation*}
whereas if $k\leq \ell-k$
\begin{equation*}\I_{(c_{1}-k,c_{2}-k,c_{3}-k)} = \left\{\begin{array}{ll}
q-3                 & \text{if $k=1$;}\\
(q-1)(q-2)q^{k-2}   & \text{if $k>1$.}
\end{array}\right.\end{equation*}
\end{lemma}
\begin{proof}
Let $\aa=(c_{1}-i,c_{2}-i,c_{3}-i)$ and note that $a_{1}+a_{2}= a_{3}$ if and only if $i=k$.

First suppose that $i\neq k$ and so $a_{1}+a_{2}\neq a_{3}$.  If $i=0$ then $\aa(\cc,\cc)=0$ and   $\aa\notin\TT_{\cc_{I},\cc_{J}}$ for $I,J\subseteq S$ not both empty.  This therefore implies that $\I_{\aa} = |\XX_{\cc,\cc}^{\aa}| = 1$.  For $i>0$ we have $\aa(\cc,\cc)=i$ and for $I,J\subseteq S$ not both empty $\aa\in \TT_{\cc_{I},\cc_{J}}$ with $\aa(\cc_{I},\cc_{J})=i-1$.  Consequently,
\begin{equation*}\I_{\aa} = |\XX_{\cc,\cc}^{\aa}| + \sum_{\text{other}\ I,J} (-1)^{|I|+|J|} |\XX_{\cc_{I},\cc_{J}}^{\aa}|  = |\XX_{\cc,\cc}^{\aa}| - |\XX_{\cc_{S},\cc_{S}}^{\aa}|.\end{equation*}
When $i=1$ we obtain
$\I_{\aa} = (q-1)-1=q-2$
whereas for $i>1$ this gives
$\I_{\aa}  = (q-1)q^{i-1} - (q-1)q^{i-2} = (q-1)^2q^{i-2}.$

Now suppose that $i=k$ and so $a_{1}+a_{2}=a_{3}$.  Again we see that $\aa(\cc,\cc)=i$ and $\aa\in \TT_{\cc_{I},\cc_{J}}$ with $\aa(\cc_{I},\cc_{J})=i-1$ for $I,J\subseteq S$ not both empty.  However, in this case  $\aa(\cc_{I},\cc_{J})'=i$ for $I\subseteq \{2\}$ and $J\subseteq \{1\}$ with $\aa(\cc_{I},\cc_{J})'=i-1$ otherwise.  Thus
\begin{eqnarray*}
\I_{\aa} &=& |\XX_{\cc,\cc}^{\aa}| - |\XX_{\cc_{\{1\}},\cc}^{\aa}| - |\XX_{\cc,\cc_{\{2\}}}^{\aa}| + |\XX_{\cc_{\{1\}},\cc_{\{2\}}}^{\aa}|
+ \sum_{\text{other}\ I,J} (-1)^{|I|+|J|} |\XX_{\cc_{I},\cc_{J}}^{\aa}|\\
&=& |\XX_{\cc,\cc}^{\aa}| - |\XX_{\cc_{\{1\}},\cc_{\{2\}}}^{\aa}|.
\end{eqnarray*}
When $k=1$ this gives $\I_{\aa}=(q-1)-2=q-3$ and when $k>1$ we get $\I_{\aa} = (q-1)q^{k-1} - 2(q-1)q^{k-2} = (q-1)(q-2)q^{k-1}$.
\end{proof}

\begin{lemma}\label{lem:ia2}
If $c_{1}<d_{1}$ and $k=k'\leq \lfloor\min\{\ell,\ell'\}/2\rfloor$, then
\begin{equation*}\I_{(c_{1}-k,d_{2}-k,c_{3}-k)} = \left\{\begin{array}{ll}
q-2            & \text{if $k=1$;}\\
(q-1)^2q^{k-2} & \text{if $k>1$.}
\end{array}
\right.\end{equation*}
\end{lemma}
\begin{proof}
Let $\aa=(c_{1}-k,d_{2}-k,c_{3}-k)$ and recall that $a_{1}+a_{2}\neq a_{3}$.  Then $\aa(\cc,\dd)=k$ and $\aa(\cc_{I},\dd_{I})=k-1$ for $I,J\subseteq S$ not both empty so the result follows by the argument for the first part of the previous Lemma.
\end{proof}

\begin{proof}[Proof of Theorem \ref{thm:threedesc}]
By Lemmas \ref{lem:iasum} and \ref{lem:ia} we see that if $k=1$ then
\begin{equation*}\I(V_{\cc},V_{\cc}) = \I_{(c_{1},c_{2},c_{3})} + \I_{(c_{1}-1,c_{2}-1,c_{3}-1)} = 1 + (q-3) = q-2\end{equation*}
and similarly when $1<k\leq \ell-k$
\begin{eqnarray*}
\I(V_{\cc},V_{\cc})
&=& \sum_{i=0}^{k} \I_{(c_{1}-i,c_{2}-i,c_{3}-i)}  \\
&=& 1 + (q-2) + (q-1)^2q + \cdots + (q-1)^2q^{k-3} + (q-1)(q-2)^{k-2} \\
&=& (q-1)^2q^{k-2}.
\end{eqnarray*}
However, if $\ell-k<k<\ell-1$ then
\begin{eqnarray*}
\I(V_{\cc},V_{\cc})
&=& \sum_{i=0}^{\ell-k} \I_{(c_{1}-i,c_{2}-i,c_{3}-i)}  \\
&=& 1 + (q-2) + (q-1)^2q + \cdots + (q-1)^2q^{\ell-k-3} + (q-1)^2q^{\ell-k-2} \\
&=& (q-1)q^{\ell-k-1}.
\end{eqnarray*}
and when $k=\ell-1$, with $\ell> 2$,
\begin{equation*}\I(V_{\cc},V_{\cc}) = \I_{(c_{1},c_{2},c_{3})} + \I_{(c_{1}-1,c_{2}-1,c_{3}-1)} = 1 + (q-2) = q-1.\end{equation*}

Finally, if $\cc$ and $\dd$ have $c_{3}=d_{3}$, $k=k'$ and
$k\leq\lfloor\min\{\ell,\ell'\}/2\rfloor$, then by the calculations above and Lemmas \ref{lem:iasum} and \ref{lem:ia2}
\begin{equation*}\I(V_{\cc},V_{\dd})=\I(V_{\cc},V_{\cc}) = \I(V_{\dd},V_{\dd}).\end{equation*}
Hence $V_{\cc}$ and $V_{\dd}$ must be equivalent.
\end{proof}

It should be noted that in proving the reducibility of $V_{\cc}$, we have discovered a certain amount of information about its decomposition.  Let $\ii=(i,i,i)$ for $0<i\leq \min\{k,\ell-k\}$ and consider the representation  $V^{i}_{\cc} = U^{i}_{\cc} / \sum_{\cc\prec \dd \preceq \cc-\ii} U^{i}_{\dd}$ where $U^{i}_{\dd} = \Ind^{C_{\cc-\ii}}_{C_{\dd}} 1$.  Clearly, $V_{\cc} = \Ind_{C_{\cc-\ii}}^{K} V^i_{\cc}$ and we may use the results of Section~\ref{sec:defn} to show that
\begin{equation*}\I(V^{i}_{\cc},V^{i}_{\cc}) = \sum_{I,J\subseteq S} (-1)^{|I|+|J|} | C_{\cc_{I}} \backslash C_{\cc-\ii} / C_{\cc_{J}}|.\end{equation*}
However, the $(C_{\cc_{I}},C_{\cc_{J}})$-double coset representatives in $C_{\cc-\ii}$ are precisely the $t_{\aa,x}$ in $\RR^{1}_{\cc_{I},\cc_{J}}$ which have $\aa\succeq \cc-\ii$.  Thus
\begin{equation*}\I(V^{i}_{\cc},V^{i}_{\cc}) = \sum_{\aa\succeq \cc-\ii} \I_{\aa},\end{equation*}
where $\I_{\aa}$ is as before, and Lemmas~\ref{lem:iasum} and \ref{lem:ia} immediately imply the following.

\begin{proposition}
Let $0 < i \leq \min\{k,\ell-k\}$.  For $i\neq k$ we have
\begin{equation*}\I(V^{i}_{\cc},V^{i}_{\cc}) = \left\{\begin{array}{ll}
(q-1) & \text{if $i=1$}; \\
(q-1)q^{i-1} & \text{if $i>1$}
\end{array} \right.
\end{equation*}
whereas if $k\leq \ell-k$
\begin{equation*}
\I(V^{k}_{\cc},V^{k}_{\cc}) = \left\{\begin{array}{ll}
(q-2) & \text{if $k=1$}; \\
(q-1)^2q^{k-2} & \text{if $k>1$}.
\end{array} \right.\end{equation*}
\end{proposition}

In particular, this means that the irreducible constituents of $V_{\cc}$ are induced from the irreducible constituents of $V_{\cc}^{\min\{k,\ell-k\}}$.

\section{Application to Steinberg representations}

In \cite{lees}, Lees defined a virtual representation $S_{r}$ of $\mathrm{GL}(n,\R/\P^{r})$ which possessed properties that were similar to those of the Steinberg representation of $\mathrm{GL}(n,\mathfrak{f})$.  Further, he stated without proof that $S_{r}$ was in fact a subrepresentation of the permutation representation over the subgroup of upper triangular matrices.  Although this is the case for $n=2$ and for $n=3$ with $r\leq 2$, we will show that $S_{r}$ is not a true representation for $r>2$.

When $n=3$, and pulling back to $K$, the expression of $S_{r}$ as an alternating sum of permutation representations reduces to
\begin{equation} \label{eq:sr}
[S_{r}] = \sum_{c_{1},c_{2}=0}^{r} (-1)^{c_{1}+c_{2}}[U_{(c_{1},c_{2},\max\{c_{1},c_{2}\})}].
\end{equation}
If $r=0$ then we obtain the trivial representation $S_{0}=V_{(0,0,0)}$ and
$r=1$ produces $S_{1} = V_{(1,1,1)}$, the Steinberg representation of $\mathrm{GL}(3,\mathfrak{f})$ pulled back to $K$.  More generally, $S_{r}$ can be constructed inductively in the following way.

\begin{lemma}\label{lem:sr} Let $r\geq 2$, then
\begin{equation} \label{eq:sr2}
[S_{r}] = [S_{r-2}] + \sum_{\cc} (-1)^{r-c_{1}} [V_{\cc}]
\end{equation}
where the sum runs over all triples $\cc=(c_{1},c_{2},r)\in \TT$ with $c_{1}\equiv c_{2}\ (\mathrm{mod}\ 2)$.
\end{lemma}
\begin{proof}
This can easily be seen by comparing the coefficients of $U_{\cc}$ in (\ref{eq:sr}) and (\ref{eq:sr2}) for each $\cc\preceq(r,r,r)$.
\end{proof}

In particular, Lemma \ref{lem:sr} implies that
\begin{equation*}[S_{2}] = [V_{(2,2,2)}]-[V_{(1,1,2)}]+[V_{(0,2,2)}]+[V_{(2,0,2)}] + [V_{(0,0,0)}].\end{equation*}
Further, $V_{(1,1,2)}$,  $V_{(0,2,2)}$ and $V_{(2,0,2)}$ are all equivalent so this is actually the sum of three irreducible representations $S_{2} \simeq V_{(2,2,2)}+V_{(1,1,2)} + V_{(0,0,0)}$.  However, when $r>2$ we see that $V_{(r-1,r-1,r)}$ still appears with coefficient $-1$ in (\ref{eq:sr2}) and in this case $V_{(r-1,r-1,r)}$ has no constituents in common with any other component $V_{\cc}$.  Hence it cannot cancel with any other term in (\ref{eq:sr2}) and we have shown the following.

\begin{proposition}
$S_{r}$ is a true representation if and only if $r\leq 2$.
\end{proposition}

\end{document}